\documentclass[12pt]{amsart}
\usepackage{fullpage}
\usepackage{tabu}
\usepackage{amssymb} 
\usepackage{amsmath} 
\usepackage{amscd}
\usepackage{amsbsy}
\usepackage{comment, enumerate}
\usepackage[matrix,arrow]{xy}
\usepackage[
        colorlinks, citecolor=blue,
        backref,
]{hyperref}
\usepackage{mathrsfs}
\usepackage{color}
\usepackage{mathtools,caption}
\usepackage{todonotes}
\usepackage{stmaryrd}
\usepackage{tikz-cd} 
\usepackage{enumerate}
\usepackage{comment}
\usepackage[T1]{fontenc}

\DeclareMathOperator{\Norm}{Norm}
\DeclareMathOperator{\Tr}{Tr}
\DeclareMathOperator{\ch}{\mathrm{char}}
\DeclareMathOperator{\Gal}{Gal}
\DeclareMathOperator{\Frobp}{Frob_{\fp}}
\DeclareMathOperator{\condexp}{condexp}
\DeclareMathOperator{\PS}{PS}
\DeclareMathOperator{\St}{St}
\DeclareMathOperator{\Res}{Result}
\DeclareMathOperator{\Ind}{Ind}
\DeclareMathOperator{\Art}{Art}
\DeclareMathOperator{\Frob}{Frob}
\DeclareMathOperator{\sep}{sep}
\DeclareMathOperator{\End}{End}
\DeclareMathOperator{\GL}{GL}
\DeclareMathOperator{\SL}{SL}
\DeclareMathOperator{\PSL}{PSL}
\DeclareMathOperator{\PGL}{PGL}

\newcommand{\mf}{\mathfrak}
\newcommand{\spec}{\mathrm{Spec}}

\newcommand{\calO}{\mathcal{O}}
\newcommand{\calC}{\mathcal{C}}

\newcommand{\rhobar}{{\overline{\rho}}}

\newcommand{\Z}{\mathbb{Z}}
\newcommand{\Q}{\mathbb{Q}}
\newcommand{\R}{\mathbb{R}}
\newcommand{\C}{\mathbb{C}}
\newcommand{\A}{\mathbb{A}}
\newcommand{\F}{\mathbb{F}}
\newcommand{\Fp}{\mathfrak{p}}
\newcommand{\fp}{\mathfrak{p}}

\newcommand{\vfq}{v_{\fq}}
\newcommand{\fN}{\mathfrak{N}}

\newcommand{\Fq}{\mathfrak{q}}
\newcommand{\Fr}{\mathfrak{r}}

\newcommand{\Qbar}{{\overline{\Q}}}
\newcommand{\fq}{\mathfrak{q}}
\newcommand{\fqf}{\fq_{5}}

\newcommand{\fpf}{\fp_5}
\newcommand{\n}{\mathfrak{n}}
\newcommand{\OK}{\mathcal{O}_K}
\newcommand{\OL}{\mathcal{O}_L}

\newcommand{\Kp}{K_{\mf{p}}}
\newcommand{\Kqf}{K_{\fqf}}
\newcommand{\Kq}{K_{\mf{q}}}

\newcommand{\Kqun}{\Kq^{un}}
\newcommand{\bKp}{\overline{K}_\fp}
\newcommand{\bKq}{\overline{K}_\fq}

\newcommand{\IKq}{I_{\Kq}}
\newcommand{\OKp}{\mathcal{O}_{\Kp}}
\newcommand{\OKq}{\mathcal{O}_{\Kq}}
\newcommand{\OKqf}{\mathcal{O}_{\Kqf}}

\newcommand{\pifq}{\pi_{\fq}}

\newcommand{\WKq}{W_{\Kq}}
\newcommand{\WKqab}{W^{\textrm{ab}}_{\Kq}}
\newcommand{\GalKq}{G_{\Kq}}
\newcommand{\Qp}{\Q_p}
\newcommand{\bQp}{\overline{\Q}_p}
\newcommand{\Zp}{\Z_p}

\newcommand{\Ffp}{\mathbb{F}_{\mf{p}}}
\newcommand{\Ffq}{\mathbb{F}_{\mf{q}}}
\newcommand{\FFp}{\mathbb{F}_p}

\newcommand{\bFFp}{\overline{\mathbb{F}}_p}

\newcommand{\bFfq}{\overline{\mathbb{F}}_{\mf{q}}}
\newcommand{\rhoJp}{\rho_{J,\fp}}
\newcommand{\rhoJpf}{\rho_{J,\fpf}}
\newcommand{\brhoJp}{\overline{\rho}_{J,\fp}}
\newcommand{\brhoJpf}{\overline{\rho}_{J,\fpf}}

\newcommand{\rhoHilfp}{\rho_{f, \fp}}

\numberwithin{equation}{section}

\newtheorem{theorem}{Theorem}[section]
\newtheorem{lemma}[theorem]{Lemma}
\newtheorem{corollary}[theorem]{Corollary}
\newtheorem{proposition}[theorem]{Proposition}

\theoremstyle{definition}
\newtheorem{definition}[theorem]{Definition}

\newtheorem{conjecture}[theorem]{Conjecture}
\newtheorem{example}[theorem]{Example}

\theoremstyle{remark}
\newtheorem{remark}[theorem]{Remark}

\begin{document}


\thispagestyle{empty}


This is a slightly revised version of: 

\medskip

\sloppy I. Chen and A. Koutsianas. \textit{Darmon's Program: A survey}. Proceedings of a series of trimester programs on Triangle groups, Belyi uniformization, and Modularity, IInd Trimester, Jan 2022 - Apr 2022, Modularity and Generalized Fermat's Equation, Bhaskaracharya Pratishthana Educational and Research Institute, Pune, India. Accepted, July 25, 2023. 

\medskip

Publication of the proceedings has been delayed so we have made this version available in the meantime.

\newpage

\textbf{Notations and Conventions}

\vspace*{2ex}
\noindent $\Z$ \quad  ring of integers \\
$\Q$ \quad  field of rational numbers \\
$\R$ \quad field of real numbers \\ 
$\C$ \quad field of complex numbers \\
$\Q(\zeta_n)$ \quad the $n$th cyclotomic field\\
$\Q(\zeta_n)^+$ \quad the maximal totally real subfield of $\Q(\zeta_n)$\\
$\mathbb{F}$ \quad field \\
$\mathbb{F}[X_1,X_2,\ldots,X_n]$ \quad ring of polynomials in $n$-variables with $\mathbb{F}$-coefficients \\
$\mathbb{F}(X_1,X_2,\ldots,X_n)$ \quad fraction field of $\mathbb{F}[X_1,X_2,\ldots,X_n]$ \\
$\A^n(\mathbb{F})$ \quad affine space over $\mathbb{F}$ respectively\\
$\mathbb{P}^n(\mathbb{F})$ \quad projective spaces over $\mathbb{F}$\\
$R$ \quad ring \\
$\spec(R)$ \quad the set of prime ideals in  ring $R$ \\
$R_{\mf{p}}$ \quad the ring of integers of the completion of the fraction field of $R$ for any \\
\hspace*{4.5ex} $\mf{p}\in\spec(R)$ as well as the localization of $R$ at $\mf{p}$\\
$A$ \quad Dedekind domain \\
$K$ \quad fraction field of $A$ \\
$\OK$ \quad ring of integers of $K$ \\
$L$ \quad finite extension of $K$ \\
$\OL$ \quad ring of integers of $L$ \\
$\mf{p},\mf{q},\mf{m},\mf{n}$ \quad prime ideals \\
$\mathbb{F}_p$ \quad the finite field of $p$ elements\\
$\Qp$ \quad  completion of $\Q$ at the prime $p$ \\
$\Zp$ \quad  ring of $p$-adic integers \\
$\Ffp$ \quad the residue field of $\mf{p}$ \\
$\Kp$ \quad completion of $K$ at a prime $\mf{p}$ \\
$\OKp$ \quad ring of integers of $\Kp$ \\
$v_\fp$ \quad the normalized valuation of $\OKp$ \\
$\zeta_K(s)$ \quad Dedekind zeta function associated to $K$ \\
$\Frob_{\mf{p}}$ \quad Frobenius element \\
$D(\mf{P}|\mf{p})$ \quad decomposition group for prime $\mf{P}$ over $\mf{p}$ \\
$I(\mf{P}|\mf{p})$ \quad inertia groups for prime $\mf{P}$ over $\mf{p}$ \\
$D_{\mf{p}}$ \quad absolute decomposition groups at a prime \\
$I_{\mf{p}}$ \quad absolute Inertia groups at a prime \\
$K^{\sep}$ \quad   a fixed separable closure of $K$ \\
$G_K=\Gal(\sep{K}/K)$ \quad absolute Galois group of $K$ \\
$\overline{K}$ \quad a fixed algebraic closure of $K$ \\
$\mathrm{char}(K)$ \quad characteristic of the field $K$

\newpage

\title{Darmon's Program: A survey}


\author{Imin Chen}
\address{Department of Mathematics, Simon Fraser University\\
Burnaby, BC V5A 1S6, Canada.} 
\email{ichen@sfu.ca}

\author{Angelos Koutsianas}
\address{Department of Mathematics, Aristotle University of Thessaloniki\\
School of Science, 3rd floor, office 17, 54124, Thessaloniki, Greece} 
\email{akoutsianas@math.auth.gr}

\date{}

\begin{abstract}
We give an overview of Darmon's program for resolving families of generalized Fermat equations with one varying exponent and survey what is currently known about this approach based on recent work of Billerey-Chen-Dieulefait-Freitas and Chen-Koutsianas. Additionally, we provide background material which is helpful to understand and apply the methods developed in these recent works. In particular, we explain the basic strategy for and simplified examples of each of the steps that is required in order to resolve a family of generalized Fermat equations.
\end{abstract}

\maketitle

\tableofcontents

\section{The problem: generalized Fermat equations}

Consider the generalized Fermat equation
\begin{equation}\label{general-equ}
   x^p + y^q = z^r,
\end{equation}
of \textit{signature} $(p,q,r)$ where $p, q, r \in \mathbb{N}$. We say that a solution $(a,b,c) \in \Z^3$ to \eqref{general-equ} is \textit{primitive} if $\text{gcd}(a,b,c) = 1$ and \textit{trivial} if $abc = 0$. 

The following has been conjectured by several people, including Beal, Granville, Tijdeman-Zagier, and there is currently a prize offered for its resolution \cite{Mauldin97}.

\begin{conjecture}[The ``Beal Prize'' Conjecture]\label{beal-conj}
If $p, q, r \ge 3$, then there are no non-trivial primitive solutions to $x^p + y^q = z^r$.
\end{conjecture}

We remark, without the condition of primitivity, there are in general non-trivial solutions to \eqref{general-equ} (see for instance \cite{beal-survey}).  If the signature $(p,q,r)$ is fixed and 
\begin{equation}\label{hyperbolic}
  \frac{1}{p} + \frac{1}{q} + \frac{1}{r} < 1,
\end{equation}
then Darmon-Granville \cite{Darmon-Granville} have shown that \eqref{general-equ} has only finitely primitive solutions by a reduction to Faltings' Theorem on the finiteness of rational points of smooth curves of genus $\ge 2$ over a number field, where the reduction uses a version of the Chevalley-Weil Theorem \cite{Beckmann}. 

It is worth mentioning that the ABC Conjecture implies that there exists a positive integer $n_0$ such that \eqref{general-equ} has no non-trivial primitive solutions for $p,q,r\geq n_0$ \cite{Mauldin97}. Therefore, from the result of Darmon-Granville and the ABC Conjecture we get that there are only finitely many counterexamples to Beal's Conjecture. In fact, it is conjectured \cite{Darmon-epsilon} that the primitive solutions arising from the following identities,
\begin{gather}\label{fermat-catalan}
1^n + 2^3 = 3^2, \\
\notag 2^5+7^2 = 3^4, 7^3+13^2 = 2^9, 2^7+17^3 = 71^2, 3^5+11^4 = 122^2, \\
\notag 17^7+76271^3 = 21063928^2, 1414^3+2213459^2 = 65^7, 9262^3+15312283^2 = 113^7, \\
\notag 43^8+96222^3 = 30042907^2, \text{ and }33^8+1549034^2 = 15613^3,
\end{gather}
are the only ones to \eqref{general-equ} among all signatures $(p,q,r)$ satisfying \eqref{hyperbolic}.

To resolve \eqref{general-equ} for integers $p, q, r \ge 3$, it is is sufficient to resolve \eqref{general-equ} for $p, q, r \ge 3$, where each of $p, q, r$ is either an odd prime or $4$.

The special case of $r = q = p$ is Fermat's Last Theorem which was proven in \cite{Wiles, Taylor-Wiles} using Galois representations and modularity. We refer to an approach by Galois representations and modularity as the \textit{modular method}.

The modular method for resolving a Diophantine equation such as \eqref{general-equ} consists of carrying out the following steps.

\begin{enumerate}
\item {\bf Construct a Frey elliptic curve.} Attach an elliptic curve $E/K$ to a putative solution of a Diophantine equation, where $K$ is some totally real field. Compute the discriminant $\Delta(E)$ and conductor $N(E)$ of $E/K$.

\item {\bf Modularity.} Prove modularity of $E/K$ in the sense that $\rho_{E,p} \simeq \rho_{f,\mathfrak{P}}$ for some Hilbert newform $f$, where $\rho_{E,p}$ and $\rho_{f,\mathfrak{P}}$ are the $p$-adic representations attached to $E$ and $f$, respectively.

\item {\bf Irreducibility.} Prove irreducibility of 
$\overline{\rho}_{E,p}$, the mod $p$ Galois representation attached to $E/K$.

\item {\bf Level lowering.} Conclude that
$\overline{\rho}_{E,p} \simeq \rhobar_{f,\mathfrak{P}}$ where $f$ is a Hilbert
newform over $K$ of parallel weight 2, trivial character, and Serre level $N(\overline{\rho}_{E,p})$ which has finitely many possible values. The vector space of such Hilbert newforms is finite in dimension.
\item {\bf Contradiction.} 
Compute all the Hilbert newforms predicted in the previous step; then, for each computed newform~$f$ and $\mathfrak{P} \mid p$ in its field of coefficients, show that $\overline{\rho}_{E,p} \not\simeq \overline{\rho}_{f,\mathfrak{P}}$. If we succeed in doing so, we say that \textit{the newform $f$ has been eliminated}. A typical method for elimination is to show the two representations do not have the same trace at an unramified Frobenius element $\Frob{\Fq}$. This is a method which is local at the prime $\Fq$ of $K$.
\end{enumerate}

Every step of the modular method runs into difficulties when applied to general signature. In \cite{DarmonDuke}, Darmon described an approach which can be used to address some (but not all) of the difficulties encountered for general signature. 

In this survey, we will give an overview of his initial framework and described recent advances on his program contained in \cite{BCDF2, Chen-2022-xhyper, ChenKoutsianas1}. It is intended to provide ``quick start overview and user guide'' to these papers and the program.

\section{Darmon's program to the rescue!}

In \cite{DarmonDuke}, Darmon described a program to show the generalized Fermat equations \eqref{general-equ} have no non-trivial primitive solutions where there is one prime exponent $p \ge 3$ which varies, using the approach of Galois representations and modularity. Darmon permutes the terms into a standard form so the signatures addressed are one of the following: $(p,p,r)$, $(r,r,p)$, or $(q,r,p)$. 

Recall an abelian variety $A/K$ defined over a number field $K$ such that $\End_K(A) \otimes \Q$ contains a number field $F$ of degree $[F:\Q] = \dim A$ is called an \textit{abelian variety of $\GL_2(F)$-type}.

The steps of the modular method are generalized and modified in Darmon's program as follows:

\smallskip

{\bf 1. Construct a Frey abelian variety.} Darmon replaces the elliptic curve $E/K$ by an abelian variety $A/K$ of $\GL_2(K)$-type, where $K$ is a totally real field. In principle, this covers all signatures $(p,q,r)$, but explicit constructions are only generally known for signatures $(p,p,r)$ and $(r,r,p)$.

\medskip

{\bf 2. Modularity.} Prove modularity of $A/K$ in the sense that $\rho_{A,\Fp} \simeq \rho_{f,\mathfrak{P}}$, where $\rho_{A,\Fp}$ and $\rho_{f,\mathfrak{P}}$ are the $p$-adic representations attached to $A$ and $f$, respectively. Due to relations between Frey abelian varieties and rigidity, it is possible to propagate modularity using a strategy of Darmon \cite{DarmonDuke, DarmonRigid}. Advances in modularity lifting theorems allows this strategy to be carried out in many cases, but the general strategy cannot be carried out due to the lack of sufficiently strong modularity lifting theorems in the residually reducible cases.

\smallskip

{\bf 3. Irreducibility.} Prove irreducibility of $\overline{\rho}_{A,\Fp}$, the mod $\Fp$ Galois representation attached to the abelian variety $A/K$ of $\GL_2(K)$-type, where $\Fp$ is a prime of $K$. There are no currently known Mazur-type results for abelian varieties of $\GL_2(K)$-type. However, under additional local conditions, it is possible to prove irreducibility.

\medskip

{\bf 4. Level lowering.} This step can be achieved by results of Fujiwara, Jarvis, Rajaei \cite{Fuj,Jarv,Raj} which generalize Ribet's level lowering results from $\Q$ to totally real fields $K$. As before, we conclude that $\overline{\rho}_{A,\Fp} \simeq \rhobar_{f,\mathfrak{P}}$ where $f$ is a Hilbert
newform over $K$ of parallel weight 2, trivial character, and Serre level $\n(\overline{\rho}_{A,\Fp})$ which has finitely many possible values.

\medskip

{\bf 5. Contradiction.} 
Compute all the Hilbert newforms predicted in the previous step; then, for each computed newform~$f$ and $\mathfrak{P} \mid p$ in its field of coefficients, show that $\overline{\rho}_{A,\Fp} \not\simeq \overline{\rho}_{f,\mathfrak{P}}$. 
 
\medskip
 
In general, Step 5 is perhaps the hardest step. For example, there may be a trivial primitive solution to \eqref{general-equ} such that the Frey abelian variety $A_0$ attached to it has the property that $\rho_{A_0,\Fp} \simeq \rho_{f_0,\mathfrak{P}}$ occurs at one of the possible Serre levels. The newform $f_0$ cannot be eliminated in such a case by local methods and the trivial primitive solution presents an \textit{obstruction} for the modular method to succeed.

\begin{remark}
In Step 5 we do not deal with the difficulties in the computation of Hilbert newforms since this is a problem of its own interest. About the problem of computing Hilbert newforms see \cite{DembeleVoight13}.
\end{remark}

One idea of Darmon's program is to make use of the following property.

\begin{proposition} (Darmon)
The Frey abelian varieties attached to trivial primitive solutions have complex multiplication (CM).
\end{proposition} 

\begin{conjecture}[Darmon]\label{big-image}
Let $K$ be totally real field. There exists a constant $C_K$ such that, for any abelian variety $A/K$ of $\GL_2(K)$-type with 
\begin{equation*}
\End_K(A)\otimes \Q = \End_{\overline{K}}(A) \otimes \Q  = K,
\end{equation*}
and all primes $\fp$ of~$K$ of norm $ > C_K$, 
we have $\SL_2(\F_\Fp) \subseteq \rhobar_{A,\fp}(G_K)$.
\end{conjecture} 

\medskip 

For $\dim A = 1$, Conjecture~\ref{big-image} is resolved in most cases by works of Mazur \cite{Mazur-isogeny} and Bilu-Parent \cite{Bilu-Parent}, with the normalizer of non-split Cartan subgroup case still outstanding (but see \cite{non-split-13, BalakrishnanDograMullerTuitmanVonk23} for recent progress). We note however, Darmon-Merel \cite{DarmonMerel} were able to prove a weaker form of the above conjecture in $\dim A = 1$ assuming $A$ has additional level $r$-structure where $r = 2, 3$, which suffices for applications to generalized Fermat equations of signature $(p,p,2)$ and $(p,p,3)$.

Conjecture~\ref{big-image} would deal with Step 3, and Step 5 for the problematic CM form, for $p$ sufficiently large.  However, this conjecture is wide open!

We can bypass obstructive trivial primitive solutions by imposing additional restrictions on the solutions considered. The results obtained show that in many cases, the only obstruction to obtaining a complete resolution of a generalized Fermat equation are trivial solutions which give rise to Frey abelian varieties with CM. For example, in the signature $(p,p,2)$ the obstruction arising by the trivial solution $(1, -1, 0)$ corresponds to the elliptic curve $E:~y^2 = x^3 - x$ which has CM over $\Q(\sqrt{-1})$ \cite{DarmonMerel}.

Finally, we end with what is another critical difficulty which was not addressed in Darmon's original program: as we have mentioned, the approach naturally resolves generalized Fermat equations in one varying exponent $p$. Keeping the other exponent(s) fixed, the space of Hilbert newforms to be eliminated grows with the size of the fixed exponent(s). 

Concretely, this has the consequence that even Conjecture~\ref{big-image} does not suffice to carry out Darmon's program because one still needs to eliminate the Hilbert newforms at the Serre level without CM. Hence, one requires new methods and computational efficiencies in addition to establishing Conjecture~\ref{big-image} to realize Darmon's program for new signatures.

The totally real fields which naturally arise in Darmon's program come from the maximal totally real subfield of a  cyclotomic field. For this, we will use the notation $\Q(\zeta_n)$ to be the extension obtain by adjoining a primitive $n$th root of unity $\zeta_n$, and $\Q(\zeta_n)^+$ to denote its maximal totally real subfield.

\section{Acknowledgements}

We would like thank the Bhaskaracharya Pratishthana Institute for the opportunity to give talks on these topics. We also thank Nicolas Billerey, Luis Dieulefait, and Nuno Freitas for their collaborations on these topics over the years, and additionally to Nuno Freitas for allowing us to include some material from his previous presentations on Darmon's program.

\section{Frey abelian varieties: explicit examples}

Let $p,q,r$ be odd primes, $t_1, t_2, t_3 \in \mathbb{P}^1(\overline{K})$, and $K = \Q(\zeta_q,\zeta_r)^+$ be the maximal totally real subfield of $\Q(\zeta_q, \zeta_r)$. A \textit{Frey representation of signature}~$(p,q,r)$ with respect to the points $(t_1,t_2, t_3)$, is a Galois representation
\begin{equation*}
  \rhobar : G_{K(t)} \rightarrow \GL_{2}(\F) \qquad \text{where } [\F : \F_p]  \text{ is finite,}
\end{equation*}
with projectivization $\mathbb{P}(\rhobar) : G_{K(t)} \rightarrow \PGL_2(\F)$ and such that 
\begin{itemize}
    \item $\rhobar |_{G_{\overline{K}(t)}}$ has trivial determinant and is irreducible;

    \item $\mathbb{P}(\rhobar)|_{G_{\overline{K}(t)}} : G_{\overline{K}(t)} \rightarrow \PSL_{2}(\F)$ is unramified outside $\left\{ t_1, t_2, t_3 \right\}$;

    \item $\mathbb{P}(\rhobar)|_{G_{\overline{K}(t)}}$ maps the inertia subgroups at $t_1, t_2, t_3$ to cyclic subgroups of $\PSL_{2}(\F)$ of order $p$, $q$, $r$, respectively.
\end{itemize}
For an explanation of inertia subgroups in the setting of function fields of characteristic $0$ see \cite{Stichtenoth09, serretopics}. One should think of $\rhobar = \rhobar(t)$ as a family of $G_K$-representations in the parameter $t \in \mathbb{P}^1$.

We say two Frey representations $\rhobar_1$ and $\rhobar_2$ with respect to the points $(t_1, t_2, t_3)$ are equivalent when $\mathbb{P}(\rhobar_1)$ and $\mathbb{P}(\rhobar_2)$ are conjugate in $\PGL_2(\overline{\F}_p)$. 

\begin{proposition}[Darmon \cite{DarmonDuke}]\label{finite-S}
Let $\rhobar$ be a Frey representation of signature $(p,q,r)$ with respect to the points $(0, 1, \infty)$. Then there exists a finite set of primes $S$ of $K$ depending on $\rhobar$ in an explicit way such that for all non-trivial primitive solutions $(a,b,c) \in \Z^3$ to $$A x^p + B b^q = C c^r,$$ the representation $\rhobar(A a^p/Cc^r)$
has a quadratic twist which is unramified outside $S$.
\end{proposition}

Note the set of primes $S$ is in general dependent on $\rhobar$ and hence on $p, q, r, A, B, C$. The method of proof of Proposition~\ref{finite-S}, which uses a variant of the Chevalley-Weil Theorem, is not strong enough to show that $S$ is independent of the varying exponent $p$.

Let $\rhobar : G_{K(t)} \rightarrow \GL_2(\F)$ be a Frey representation of signature $(p,q,r)$ with respect to $(0,1, \infty)$. For $j = 0, 1, \infty$, let $\sigma_j \in \PSL_2(\F)$ be the image of $\mathbb{P}(\rhobar) \mid_{G_{\overline{K}(t)}}$ of a generator of inertia at $j$, respectively, chosen so that $\sigma_0 \sigma_1 \sigma_\infty = 1$.

\medskip

For each $t = 0, 1, \infty$, there is a unique lift $\tilde{\sigma}_t \in \SL_2(\F)$ of $\sigma_t$. We say that $\rhobar$ is \textit{odd} if $\tilde{\sigma}_0 \tilde{\sigma}_1 \tilde{\sigma}_\infty = -1$, and \textit{even} if $\tilde{\sigma}_0 \tilde{\sigma}_1 \tilde{\sigma}_\infty = 1$.

\subsection{Signature \texorpdfstring{$(p,p,p)$}{(p,p,p)}}
The Legendre family
\begin{equation}\label{leg-family}
    L(t):~y^2 = x(x-1)(x-t),
\end{equation}
is the universal family of elliptic curves with a basis for $2$-torsion. It gives rise to Frey representations of signature $(p,p,p)$ with respect to~$\{0,1,\infty\}$ via $\rhobar_{L,p}$, the mod $p$ representation of $L = L(t)$.

\begin{theorem}[Hecke \cite{Hecke35}] 
There is a unique Frey representation of signature $(p,p,p)$ with respect to $(0,1,\infty)$.
It arises from the Legendre family $L(t)$ as described in \eqref{leg-family} and is odd.
\end{theorem}

Let $(a,b,c) \in \Z^3$ be a non-trivial primitive solution to 
\begin{equation}
  x^p + y^p = z^p.
\end{equation}
Then $L(a^p/c^p)$ is a quadratic twist of
\begin{equation*}
y^2 = x(x-a^p)(x+b^p),
\end{equation*}
whose mod~$p$ representation is unramified outside~$2$. This follows from Proposition \ref{finite-E}. Note Proposition~\ref{finite-S} would give a set of ramified primes which includes $p$, illustrating a key and distinctive feature of the modular method for Diophantine applications.

\subsection{Signature \texorpdfstring{$(p,p,r)$}{(p,p,r)}}\label{sec:p_p_r_Frey_curves}

Let $K = \Q(\zeta_r)^+$ with ring of integers~$\calO_K$. Set $\omega_j = \zeta_r^j + \zeta_r^{-j}$ and let
\begin{equation*}
    g(x) = \prod_{j=1}^{\frac{r-1}{2}} (x + \omega_j),
\end{equation*}
be the minimal polynomial of  $-\omega_1$. We define
\begin{equation*}
    f(x) = xg(x^2 - 2) = g(-x)^2(x-2) + 2 = g(x)^2(x+2) - 2.
\end{equation*}

For $r \geq 3$, consider the hyperelliptic curves over $\Q(t)$
\begin{align*}
 C_r^-(t) \; : \; y^2 & = xg(x^2-2) + 2 - 4t, \\
 C_r^+(t) \; : \; y^2 & = (x+2)(x g(x^2-2) + 2-4t),
\end{align*}
and write $J_r^\pm(t)$ for their Jacobians over $\Q(t)$.

\begin{theorem}[Darmon \cite{DarmonDuke}, Tautz-Top-Verberkmoes \cite{TTV}]\label{thm:GL_ppr}
Let $\fp \mid p$ be a prime in $\calO_K$. 
Then $J_r^\pm(t)$ is of $\GL_2(K)$-type over $K$ and the mod~$\fp$ representation 
\begin{equation*}
\rhobar_{J_r^\pm(t), \fp} : G_{K(t)} \rightarrow \GL_2(\F_{\fp}),
\end{equation*}
is a Frey representation of signature~$(p,p,r)$ with respect to~$\{0,1,\infty\}$, which is odd for $J^-_r(t)$, and even for $J^+_r(t)$.
\end{theorem}

Let $(a,b,c) \in \Z^3$ be a non-trivial primitive solution to
\begin{equation}
  x^p + y^p = z^r.
\end{equation}

\begin{lemma}[Darmon \cite{DarmonDuke}]\label{lem:Frey-model}
Let
\begin{align}
  C_r^-(a,b,c): & \quad y^2 = c^r f(x/c) - 2(a^p-b^p), \\
  C_r^+(a,b,c): & \quad y^2 = (x+2c) (c^r f(x/c) - 2 (a^p - b^p)).
\end{align}
Then
\begin{enumerate}
\item $C_r^-(a,b,c)$ is isomorphic to a twist of $C_r^-(t)$ over $\Q$ where $t = a^p/c^r$,
\item $C_r^+(a,b,c)$ is isomorphic to $C_r^+(t)$ over $\Q$ where $t = a^p/c^r$.
\end{enumerate}
\end{lemma}

\begin{table}[h]
\begin{equation*}
\begin{array}{|c|c|}
\hline
r & C_r^-(a,b,c) \\
\hline
3 &  y^2 = x^3 - 3 c^2 x - 2 (a^p - b^p) \\
5 &  y^2 = x^5 - 5 c^2 x^3 + 5 c^4 x - 2 (a^p - b^p) \\
7 &  y^2 = x^7 - 7 c^2 x^5 + 14 c^4 x^3 - 7 c^6 x - 2 (a^p - b^p) \\
11 & y^2 = x^{11} - 11 c^2 x^9 + 44 c^4 x^7 - 77 c^6 x^5 + 55 c^8 x^3 - 11 c^{10} x - 2(a^p - b^p) \\
\hline
\end{array}
\end{equation*}
\caption{Examples of $C_r^-(a,b,c)$ for small values of $r$.}
\end{table}

\begin{table}[h]
\begin{equation*}
\begin{array}{|c|c|}
\hline
r & C_r^+(a,b,c) \\
\hline
3 & y^2 = (x + 2c) (x^3 - 3 c^2 x - 2 (a^p - b^p)) \\
5 & y^2 = (x + 2c) (x^5 - 5 c^2 x^3 + 5 c^4 x - 2 (a^p - b^p)) \\
7 & y^2 = (x + 2c) (x^7 - 7 c^2 x^5 + 14 c^4 x^3 - 7 c^6 x - 2 (a^p - b^p)) \\
11 & y^2 = (x + 2c) (x^{11} - 11 c^2 x^9 + 44 c^4 x^7 - 77 c^6 x^5 + 55 c^8 x^3 - 11 c^{10} x - 2(a^p - b^p)) \\
\hline
\end{array}
\end{equation*}
\caption{Examples of $C_r^+(a,b,c)$ for small values of $r$.}
\end{table}

The discriminants of the hyperelliptic curves $C_r^-(a,b,c)$ and $C_r^+(a,b,c)$ are given by
\begin{align}
  & \Delta(C_r^-)  =  2^{4(r-1)} r^r a^{p (r-1)/2} b^{p (r-1)/2} \label{eq:disc_Cminus}\\
  & \Delta(C_r^+)  =  2^{4r} r^r a^{p (r+3)/2} b^{p (r-1)/2}.\label{eq:disc_Cplus}
\end{align}
Denote by $J_r^\pm(a,b,c)$ the Jacobian of $C_r^\pm(a,b,c)$.

\medskip

\begin{remark}
$C_r^\pm(1,-1,0)$ is non-singular and $J_r^\pm(1,-1,0)$ has complex multiplication by $\Q(\zeta_r)$ \cite[Proposition 3.7]{DarmonDuke}.
\end{remark}

\subsection{Signature \texorpdfstring{$(r,r,p)$}{(r,r,p)}}\label{sec:r_r_p_Frey_curves}

Let $K = \Q(\zeta_r)^+$ with ring of integers~$\calO_K$. In signature $(r,r,p)$, we have at our disposal many Frey abelian varieties defined over $K$:
\begin{enumerate}
    \item Darmon's Frey abelian varieties which are of $\GL_2(K)$-type \cite{DarmonDuke},

    \item Frey elliptic curves constructed by Freitas \cite{Freitas15},

    \item Kraus' hyperelliptic curve realization of Darmon's odd Frey abelian varieties.
\end{enumerate}

In \cite{DarmonDuke}, Darmon gives constructions of the Frey abelian varieties $J_{r,r}^\pm(t)$ and $J_{r,q}^\pm(t)$ of signatures $(r,r,p)$ and $(q,r,p)$. These are constructed from quotients of the Jacobians of superelliptic curves and are less explicit than the case of $(p,p,r)$. They are also called ``\textit{hypergeometric abelian varieties}'' in the literature \cite{Archinard} and have periods which are given by the classical hypergeometric functions. These have applications for producing rapidly converging series for $1/\pi$ which were first discovered by Ramanujan (see \cite{Chen-2022-chudnovsky} for instance).

In \cite{Freitas15}, Freitas introduced a number of Frey elliptic curves for signature $(r,r,p)$ which are defined over $K = \Q(\zeta_r)^+$:
\begin{itemize}
\item $E_{(a,b)}^{(k_1,k_2,k_3)}$ defined over $K$; if $r \equiv 1 \pmod 6$, then there are choices of $(k_1,k_2)$ which allow this elliptic curve to be defined over the subfield of $K$ of subdegree $3$,
\item $E_{(a,b)}^{(k_1,k_2)}$ defined over $K$; if $r \equiv 1 \pmod 4$, then there are choices of $(k_1,k_2)$ which allow this elliptic curve to be defined over the subfield of $K$ of subdegree $2$,
\item $E_{(a,b)}^{(k_1,n_2,\pm)}$ defined over $K$ which are $k$-curves, where $k$ is the subfield of $K$ of subdegree $2$; a $k$-curve means that the isogeny class of this elliptic curve is defined over $k$.
\end{itemize}
Freitas' Frey elliptic curves are constructed by factoring 
\begin{equation}
  x^r + y^r = \prod_{j=0}^{r-1} (x + \zeta_r^j y),
\end{equation}
and making a delicate scaling of quadratic factors over $K = \Q(\zeta_r)^+$ and the quadratic polynomial $(x \pm y)^2$ to form solutions to 
\begin{equation*}
   A(x,y) + B(x,y) = C(x,y),
\end{equation*} 
as polynomials in $x,y$.
Each Frey elliptic curve is then constructed from the Legendre elliptic curve 
\begin{equation*}
   Y^2 = X(X-A(x,y))(X+B(x,y)).
\end{equation*}

The existence of Freitas' Frey elliptic curves for signature $(r,r,p)$ is interesting as it is not predicted by Darmon's classification of Frey representations. This discovery has paved the way for richer results in signature $(r,r,p)$, where the use of multiple Frey abelian varieties is emerging to be a crucial ingredient in resolving new cases of the generalized Fermat equation.

In \cite{Chen-2022-xhyper}, we recall a Frey hyperelliptic curve $C_r(a,b)$ due to Kraus which has discriminant
\begin{equation}\label{Kraus-discriminant}
\Delta(C_r(a,b)) = (-1)^\frac{r-1}{2} 2^{2(r-1)} r^r (a^r + b^r)^{(r-1)}.
\end{equation}
The construction of $C_r(a,b)$ can be obtained via a transformation
from Darmon's Frey hyperelliptic curve for signature $(p,p,r)$, and in fact gives a hyperelliptic realization of Darmon's $J_{r,r}^-(t)$ via a transformation, which we now explain below.

\begin{table}
\begin{equation*}
\begin{array}{|c|c|}
\hline 
r & C_r(a,b) \\
\hline
3 & y^2 = x^3 + 3ab x + b^3-a^3, \\
5 & y^2 = x^5 + 5 ab x^3 + 5 a^2 b^2 x + b^5-a^5, \\
7 & y^2 = x^7 + 7 ab x^5 + 14 a^2 b^2 x^3 + 7 a^3 b^3 x + b^7 - a^7, \\
11 & y^2 = x^{11} + 11 ab x^9 + 44 a^2 b^2 x^7 + 77 a^3 b^3 x^5 + 55 a^4 b^4 x^3 + 11 a^5 b^5 x + b^{11} - a^{11} \\
\hline
\end{array}
\end{equation*}
\caption{Examples of $C_r(a,b)$ for small values of $r$.}
\end{table}

Kraus' hyperelliptic curve $C_r(a,b)$ is constructed from a hyperelliptic curve $C_r(s)$ by specializing to
\begin{equation}
  s = \frac{b^r - a^r}{\sqrt{ab}^r},
\end{equation} 
and twisting by $\sqrt{ab}$. It is shown in \cite{Chen-2022-xhyper} the Jacobian $J_r(a,b)$ of $C_r(a,b)$ is of $\GL_2(K)$-type. This is established by first showing it for $J_r(s)$ and then transferring the property to $J_r(a,b)$.

Let $K(s,t)$ be the function field defined by the relation
\begin{equation}
  \frac{1}{s^2+4} = t(1-t).
\end{equation}
Define
\begin{equation}
  \alpha s = 2 t - 1,
\end{equation}
so that $\alpha$ is a square root of $t(1-t)$.

It is shown in \cite{Chen-2022-xhyper} that the quadratic twist by $\alpha$ of the base change of $C_r(s)$ to $K(s,t)$ has a model $C_r'(t)$ defined over $K(t)$. Let $J_r'(t)$ be the Jacobian of $C_r'(t)$. Then by \cite{Chen-2022-xhyper} we have that $J_r'(t)$ is of $\GL_2(K)$-type and 
\begin{equation}
  \rhobar_{J_r'(t),\Fp} : G_{K(t)} \rightarrow \GL_2(\F_\Fp),
\end{equation}
is an odd Frey representation of signature $(r,r,p)$ with respect to the points $(0,1,\infty)$.

\begin{table}
\begin{equation*}
\begin{array}{|c|c|c|}
\hline 
r & C_r'(t) & \Delta(C_r'(t)) \\
\hline
3 & y^2 = x^3 -3 (t-1) t x + (t-1) t (2 t-1) & - 2^2 3^3 t^2 (t-1)^2 \\
5 & y^2 = x^5 -5 (t-1) t x^3 + 5 (t-1)^2 t^2 x -(t-1)^2 t^2 (2 t-1) & 2^4 5^5 t^8 (t-1)^8 \\
7 & y^2 =  x^7 -7 (t-1) t x^5  + 14 (t-1)^2 t^2 x^3 -7 (t-1)^3 t^3 x & - 2^6 7^7 t^{18} (t-1)^{18} \\
& + (t-1)^3 t^3 (2 t-1) &  \\
11 & y^2 = x^{11} -11 (t-1) t x^9 + 44 (t-1)^2 t^2 x^7 &  -2^{10} 11^{11} t^{50} (t-1)^{50} \\
& -77 (t-1)^3 t^3 x^5 + 55 (t-1)^4 t^4 x^3 & \\
& -11 (t-1)^5 t^5 x + (t-1)^5 t^5 (2 t-1) &  \\
\hline
\end{array}
\end{equation*}
\caption{Examples of $C_r'(t)$ and $\Delta(C_r'(t))$ for small values of $r$.}
\end{table}

In order to use Kraus' Frey hyperelliptic curve $C_r(a,b,c)$, considerably more theory is needed \cite{Chen-2022-xhyper}, but there are more structures one can exploit. Freitas' Frey elliptic curves are not well-adapted for resolving generalized Fermat equations 
\begin{equation}
\label{rrp-twisted}
    A x^r + B x^r = C z^p,
\end{equation}
with non-zero fixed coefficients $A,B,C$, as for general $A, B$ we do not have a natural factorization of the left hand side of \eqref{rrp-twisted} over $\Q(\zeta_r)$. In contrast, Kraus' Frey hyperelliptic curve can be used to study equation \eqref{rrp-twisted} even if one of $A, B$ is not $1$.

For example, for $r = 5$, the ``general coefficient'' version of Kraus' Frey hyperelliptic curve is given by
\begin{equation}
  y^2 = x^5 + 5 AaBb x^3 + 5 (Aa)^2 (Bb)^2 x + (Bb^5-Aa^5) A^2 B^2.
\end{equation}
N. Billerey and his student are carrying out a more complete study of the ``general coefficient'' equation of signature $(r,r,p)$ in forthcoming work.

\subsection{Signature \texorpdfstring{$(q,r,p)$}{(q,r,p)}}

Suppose $q \not= r$. Darmon constructs Frey abelian varieties $J_{q,r}^\pm$ for signature $(q,r,p)$ which are defined over $\Q(\zeta_q, \zeta_r)^+$. These have been less well-studied in the literature and there are few explicit examples aside from the elliptic curve cases.

For signature $(2,r,p)$, one expects the Frey abelian varieties $J_{2,r}^\pm$ to be defined over $\Q(\zeta_r)^+$. The following transformations to Kraus hyperelliptic curve $C_r$ give a hyperelliptic realization of $J_{2,r}^-$: Setting $b = AB, a = (B^r-A^r)/2$ in Kraus' Frey hyperelliptic curve $C_r(A,B,C)$, we obtain a Frey hyperelliptic curve $C_{2,r}^-(a,b,c)$ with discriminant
\begin{equation*}
  \Delta(C_{2,r}^-(a,b,c)) = (-1)^\frac{r-1}{2} 2^{3(r-1)} r^r (a^2 + b^r)^\frac{r-1}{2}.
\end{equation*}

\begin{table}
    \begin{equation*}
        \begin{array}{|c|c|}
            \hline
            r & C_{2,r}^-(a,b,c) \\
            \hline
            3 & y^2 = x^3 + 3 b x + 2 a \\
            5 & y^2 = x^5 + 5 b x^3 + 5 b^2 x + 2 a \\
            7 & y^2 = x^7 + 7 b x^5 + 14 b^2 x^3 + 7 b^3 x + 2 a \\
            11 & y^2 = x^{11} + 11 b x^9 + 44 b^2 x^7 + 77 b^3 x^5 + 55 b^4 x^3 + 11 b^5 x + 2 a \\
            \hline
        \end{array}
    \end{equation*}
    \caption{Examples of $C_{2,r}^-(a,b,c)$ for small values of $r$.}
\end{table}

\subsection{Congruences and determinants}

In the previous subsections, we have seen how Frey representations arise from constructions of abelian varieties $J$ of $\GL_2(K)$-type, in particular from their mod $\Fp$-representations for $\Fp$ a prime of $K$, for signatures $(r,r,p)$ and $(p,p,r)$. When $\Fp$ lies above $p = r$, there is degeneration to signature $(p,p,p)$ \cite{DarmonDuke}.

\begin{example}
\label{relation-1}
A Frey representation of signature $(p,p,r)$ arises from $\rhobar_{J_r^\pm,\Fp}$. When $\Fp$ lies over $p = r$, $\rhobar_{J_r^\pm,\Fp}$ is a Frey representation of signature $(p,p,p)$. Hence, if $\Fr$ is the unique prime of $K = \Q(\zeta_r)^+$ lying above $r$, we have that
\begin{itemize}
    \item $\rhobar_{J_r^-,\Fr}$ is an odd Frey representation of signature $(r,r,r)$ and thus equivalent to $\rhobar_{L,r}$,

    \item $\rhobar_{J_r^+,\Fr}$ is an even Frey representation of signature $(r,r,r)$ and thus reducible.
\end{itemize}
\end{example}

\begin{example}
\label{relation-2}
A Frey representation of signature $(r,r,p)$ arises from $\rhobar_{J_{r,r}^\pm,\Fp}$. When $\Fp$ lies over $p = r$, $\rhobar_{J_{r,r}^\pm,\Fp}$ is a Frey representation of signature $(p,p,p)$. Hence, if $\Fr$ is the unique prime of $K = \Q(\zeta_r)^+$ lying above $r$, we have that
\begin{itemize}
    \item $\rhobar_{J_{r,r}^-,\Fr}$ is an odd Frey representation of signature $(r,r,r)$ and thus equivalent to $\rhobar_{L,r}$,

    \item $\rhobar_{J_{r,r}^+,\Fr}$ is an even Frey representation of signature $(r,r,r)$ and thus reducible.
\end{itemize}
\end{example}

Another example is degeneration from signature $(q,r,p)$ to signatures $(p,r,p)$ and $(q,p,p)$, and hence up to permutation of the points $(0,1,\infty)$, to signatures $(p,p,r)$ and $(p,p,q)$ \cite{DarmonDuke}.

\begin{example}
A Frey representation of signature $(q,r,p)$ arises from $\rhobar_{J_{q,r}^\pm,\Fp}$ where here $K = \Q(\zeta_q, \zeta_r)^+$. 

When $\Fp$ lies over $p = q$, $\rhobar_{J_{q,r}^\pm,\Fp}$ is a Frey representation of signature $(p,r,p)$. If $\mathfrak{Q}$ is a prime of $\Q(\zeta_q,\zeta_r)^+$ lying above $q$ and $\Fq$ is the unique prime of $\Q(\zeta_r)^+$ below $\mathfrak{Q}$, we have that
\begin{itemize}
    \item $\rhobar_{J_{q,r}^\pm,\mathfrak{Q}}$ is a Frey representation of signature $(q,r,q)$ and thus equivalent to $\rhobar_{J_r^\pm,\Fq}$ over $\Q(\zeta_r)^+$ up to permutation of the points $(0,1,\infty)$.
\end{itemize}
When $\Fp$ lies over $p = r$, $\rhobar_{J_{q,r}^\pm,\Fp}$ is a Frey representation of signature $(q,p,p)$. 
If $\mathfrak{R}$ is a prime of $\Q(\zeta_q,\zeta_r)^+$ lying above $r$ and $\Fr$ is the unique prime of $\Q(\zeta_q)^+$ below $\mathfrak{R}$, we have that
\begin{itemize}
    \item $\rhobar_{J_{q,r}^\pm,\mathfrak{R}}$ is a Frey representation of signature $(q,r,r)$ and thus equivalent to $\rhobar_{J_q^\pm,\Fr}$ over $\Q(\zeta_q)^+$ up to permutation of the points $(0,1,\infty)$.
\end{itemize}
\end{example}
An important ingredient needed is the fact that the determinants of the representations attached our Frey representations have cyclotomic character. This is used in the proof of modularity and also to ensure we deal with Hilbert newforms with trivial character in the elimination step.

To achieve this, we have the following result from \cite{Chen-2022-xhyper} and implicit in \cite{DarmonDuke}.


\begin{theorem}\label{det-cyclotomic}
Let $A = J_r^{\pm}$ or $J_{r}$. Then $\det \rho_{A,\lambda} = \chi_p$ where $\chi_p$ is the $p$-adic cyclotomic character. 
\end{theorem}

The above theorem limits the central character in the congruences discussed above to have order dividing $2$.

\begin{corollary}\label{quadratic-twist}
Let $K = \Q(\zeta_r)^+$. Then we have that
\begin{enumerate}
    \item $\rhobar_{J_r^\pm,\Fr} \simeq \rhobar_{L,r} \otimes \chi$ for some character $\chi$ of $G_K$ of order dividing $2$,

    \item $\rhobar_{J_{r},\Fr} \simeq \rhobar_{L,r} \otimes \chi$ for some character $\chi$ of $G_K$ of order dividing $2$.
\end{enumerate}
\end{corollary}

\begin{proof}
From Examples \ref{relation-1} and \ref{relation-2}, we have that
\begin{align*}
  \rhobar_{J_r^\pm,\Fr} & \simeq \rhobar_{L,r} \otimes \chi \\
  \rhobar_{J_{r},\Fr} & \simeq \rhobar_{L,r} \otimes \chi  
\end{align*}
for some central character $\chi$. Taking determinants of both sides yields
\begin{align*}
  \det \rhobar_{J_r^\pm,\Fr} & = \det \rhobar_{L,r} \cdot \chi^2 \\
  \det \rhobar_{J_{r},\Fr} & = \det \rhobar_{L,r} \cdot \chi^2.  
\end{align*}
Applying Theorem~\ref{det-cyclotomic} to $J_r^\pm$ and $J_{r}$ yields
\begin{equation*}
  \det \rhobar_{J_r^\pm,\Fr} = \det \rhobar_{J_{r},\Fr} = \det \rhobar_{L,r} = \chi_r,
\end{equation*}
from which we deduce that $\chi^2 = 1$.
\end{proof}

\section{Modularity}

Modularity theorems of Galois representations is one of the greatest achievements in mathematics the last 50 years. The modularity of elliptic curves over $\Q$ by Wiles, Wiles-Taylor and Breuil-Conrad-Diamond-Taylor \cite{Wiles, Taylor-Wiles, BreuilConradDiamondTaylor01} lead to the proof of Fermat's Last Theorem and Shimura-Taniyama conjecture. Since then important modularity theorems have been proved, for instance Serre's modularity conjecture \cite{serreconj1, serreconj2}.

In this section, we show how we can prove modularity of the Frey hyperelliptic curves using standard modularity lifting theorems.

\subsection{A modularity lifting theorem of Khare-Thorne} 

Let $K$ a totally real field with ring of integers $\OK$, $[K:\Q]=n$ and $p\geq 3$ a prime. Suppose
\begin{equation}
    \rho:~G_K\rightarrow\GL_2(\bQp),
\end{equation}
is a continuous odd Galois representation unramified at all but finitely many places. 

We say that $\rho$ is \textit{modular} if there exists a Hilbert modular form $f$ of some level $\fN\subset\OK$ and weight $(k_1, k_2, \cdots, k_n)\in(\Z_+)^n$ such that 
\begin{equation}
    \rho\simeq\rhoHilfp,
\end{equation}
where $\rhoHilfp:~G_K\rightarrow\GL_2(\bKp)$ is the continuous Galois representation attached to $f$ for a prime ideal $\fp\mid p$ of $\OK$. For \textit{geometric\footnote{See \cite[p. 193]{FontaineMazur95} for the definition of a geometric Galois representation.}} Galois representations we have the following conjecture \cite[Conjectures 6.1-6.3]{Boeckle}.

\begin{conjecture}\label{con:modularity}
Suppose $\rho:~G_K\rightarrow\GL_2(\bQp)$ is a geometric irreducible continuous Galois representation. Then, $\rho$ is modular.
\end{conjecture}

The Conjecture \eqref{con:modularity} is a very deep question and it is one of the main goals of Langlands program. Conjecture \eqref{con:modularity} has been proved in some cases, for example for elliptic curves over $\Q$ or a quadratic real field \cite{FLHS}. 

Even though we are far from proving Conjecture \eqref{con:modularity} in complete generality, there are many important modularity lifting theorems that we can use to prove that $\rho$ is modular in many practical cases. In this section we present a modularity lifting theorem by Khare-Thorne \cite[Theorem 1.1]{khareThorne}.

For details about the terminology and definitions that appear in Theorem \ref{thm:Khare_Thorne} regarding $p$-adic Galois representations and $p$-adic Hodge theory we recommend the following online material \cite{FontaineOuyang22, Buzzard12, Hong20} and \cite{Boeckle, Gee-2022}.

\begin{theorem}[Khare-Thorne \cite{khareThorne}]\label{thm:Khare_Thorne}
Let $K$ be a totally real field and $p>2$ a prime. Let $\rho:~G_K\rightarrow \GL_2(\Qbar_p)$ be a continuous representation satisfying the following conditions.
\begin{itemize}
    \item The representation $\rho$ is unramified almost everywhere.
    
    \item For each place $v\mid p$ of $K$, $\rho\mid_{G_{K_v}}$ is de Rham. For each embedding $\tau:~K\hookrightarrow\Qbar_p$, the representation above has Hodge-Tate weight $\lbrace 0,1\rbrace$.
    
    \item The residual representation $\rhobar:~G_K\rightarrow \GL_2(\overline{\F}_p)$ is irreducible when restricted to $G_{K(\zeta_p)}$ and modular.
\end{itemize}
Then $\rho$ is modular and arises from a Hilbert modular form of weight $(2,2, \cdots, 2)$.
\end{theorem}

\begin{remark}
The above theorem completes the results in \cite[Theorem 3.5.5]{Kisin} by covering a case when $p = 5$.
\end{remark}

\subsection{Example}

Let $C=C_5(a,b)$ be the Frey hyperelliptic curve of the generalized Fermat equation
\begin{equation*}
    x^5 + y^5 = 3z^p,
\end{equation*}
as we have described in Section \ref{sec:r_r_p_Frey_curves} where $(a,b,c)$ is a non-trivial primitive solution of the last equation. We denote by $J$ the Jacobian of $C$. We recall that $J/K$ is of $\GL_2(K)$-type over $K$ where $K=\Q(\zeta_5)^+$. Our goal is to prove that $\rhoJp$ is modular. For more details about the techniques we will present below see \cite[Section 8]{Chen-2022-xhyper}.

The modularity of $J/K$ is independent of the choice of the prime $\fp$ and the modularity of $\rhoJp$. Therefore, we will prove that $J/K$ is modular by proving that $\rhoJpf$ is modular using Theorem \ref{thm:Khare_Thorne}.  

In order to apply Theorem \ref{thm:Khare_Thorne} we need to prove that $\brhoJpf$ satisfies the conditions of the theorem. The first condition is an immediate consequence of Propositions \ref{prop:good_reduction_disc} and \ref{prop:good_reduction_Jacobian} and Theorem \ref{thm:Serre_Tate}. From Proposition \ref{prop:extend_GQ_absolute_irreducible}, $\brhoJpf$ satisfies third condition of Theorem \ref{thm:Khare_Thorne}, whose proof is given in \cite[Section 8]{Chen-2022-xhyper} and we restate below with the final conclusion.

\begin{proposition}\label{prop:extend_GQ_absolute_irreducible}
The representation $\brhoJpf$ extends to $G_\Q$ and is absolutely irreducible when restricted to $G_{K(\zeta_5)}$. Moreover, $\brhoJpf$ is modular.
\end{proposition}

Even though we do not delve into the details of $p$-adic Hodge theory we mention that the representation $\rhoJpf$ is de Rham (thanks to a result by Fontaine) as it is stated in \cite[Section II.5]{Berger04}. Moreover, from \cite[Th\'{e}or\`{e}me $2^\prime$]{Fontaine81}, $\rhoJpf$ has Hodge-Tate weight $\lbrace 0,1\rbrace$.

From the above, the conditions of Theorem \ref{thm:Khare_Thorne} are satisfied for $\rhoJpf$, hence we obtain the following theorem.

\begin{theorem}\label{thm:modularity_Jm}
The abelian variety $J/K$ is modular.
\end{theorem}

\section{Conductors}

The computation of conductors of Frey abelian varieties is explained in the case of the Jacobian of a Frey hyperelliptic curve. Although there is no current analogue of Tate's algorithm as for elliptic curves, in many situations, it is possible to determine conductors of Frey abelian varieties using explicit transformations guided by local inertia types.

\subsection{Preliminaries}\label{sec:hyperelliptic_preliminaries}

In this section we recall basic material about hyperelliptic curves. The main references are \cite{BornerBouwWewers2017, Lockhart, ChenKoutsianas1, Liu96}.

Let $C$ be a smooth projective hyperelliptic curve of genus $g_C\geq 2$ over a number field $K$ (for the definitions, see \cite{ChenKoutsianas1} for instance) that is birational over $K$ to the affine curve given by the equation
\begin{equation}\label{eq:long_weierstrass_model}
    F:~y^2 + h(x)y = g(x),
\end{equation}
where $g(x), h(x)\in K[x]$ such that $g(x)$ is a monic polynomial of degree $2g_C+1$, $h(x)$ has degree at most $g_C$ and the polynomial $f(x)=4g(x) + h(x)^2$ has no double roots in $\overline{K}$. When the characteristic $\ch(K)$ is odd then we can get the simpler equation of $C$
\begin{equation}\label{eq:short_weierstrass_model}
    F:~u^2 = f(x) = 4g(x) + h(x)^2,
\end{equation}
where $u=2y + h(x)$. Let $\infty\in C$ be the unique point at infinity.

Now let $K$ be a number field. A hyperelliptic equation $F$ as in \eqref{eq:long_weierstrass_model} or \eqref{eq:short_weierstrass_model} such that $g(x)$, $h(x)$, and $f(x)$ have coefficients in $\OK$ with $K(F)\simeq K(C)$ is called \textit{a hyperelliptic model of $C$}.

We define the discriminant of the model $F$ of $C$ to be
\begin{equation}
    \Delta_F=2^{4g_C}\Delta(g + h^2/4),
\end{equation}
where $\Delta(H)$ is the discriminant of $H\in K[x]$.

Let $E$ be a second hyperelliptic model of $C$ given by the equation
\begin{equation}
    E:~z^2 + H(u)z = G(u).
\end{equation}
Then the models $F$ and $E$ of $C$ are related by the transformation of the shape
\begin{equation}
    x = e^2u + r, \qquad y=e^{2g_C+1}z + t(u),
\end{equation}
where $e\in K^*$, $r\in K$ and $t(u)\in K[u]$ with $\deg(t)\leq g_C$. The discriminants between the hyperelliptic models $F$ and $E$ are related by
\begin{equation}\label{eq:discriminant_between_models}
    \Delta_F = \Delta_E e^{-4g_C(2g_C+1)}.
\end{equation}

\begin{definition}
A \textit{model} $\calC_\fq$ over $\OKq$ for a hyperelliptic curve $C$ over $\Kq$ is a $\OKq$-scheme which is proper and flat over $\OKq$ such that $\calC_{\Kq}\simeq C$ where $\calC_{\Kq}$ is the generic fiber of $\calC_\fq$.
\end{definition}
A hyperelliptic model for $C$ gives rise to a model of $C$ over $\OKq$ in a natural way by gluing together two affine hyperelliptic equations (see \cite{ChenKoutsianas1} for instance).

\begin{definition}
A model $\calC_\fq$ over $\OKq$ for a hyperelliptic curve $C$ over $\Kq$ has \textit{good reduction} if and only if its reduction mod $\vfq$ is non-singular over $\Ffq$. In addition, we say that $C$ has \textit{bad semistable reduction} if and only if its reduction mod $\vfq$ is reduced, singular, and has only ordinary double points as singularities.
\end{definition}

\begin{definition}
A hyperelliptic curve $C$ over $\Kq$ has \textit{good reduction} (resp. \textit{bad semistable reduction}) if and only if there is some model $\calC_\fq$ over $\OKq$ for $C$ which has good reduction (resp. bad semistable reduction). We say that $C$ has \textit{semistable reduction} if and only if it has good or bad semistable reduction.
\end{definition}

\begin{proposition}\label{prop:good_reduction_disc}
Let $C$ be a hyperelliptic curve over $K$. Then
$C$ has good reduction at $\fq$ if and only if $C$ has a model $\calC_\fq$ over $\OKq$ as in \eqref{eq:long_weierstrass_model} such that $\vfq(\Delta_{\calC_\fq}) = 0$.
\end{proposition}

\begin{proof}
See \cite[Proposition 3.6]{ChenKoutsianas1}.
\end{proof}

Let $\bar{C}$ be the special fiber of $C$ at $\Fq$ with respect to a hyperelliptic model. 

\begin{proposition}\label{prop:reduced_hyperelliptic}
The curve $\bar{C}$ is reduced and absolutely irreducible. The point $\infty$ reduces to a smooth point $\bar{\infty}$ of $\bar{C}$, and the affine open part $\bar{C}\setminus\{\bar{\infty}\}$ is a plane curve with equation \begin{equation*}
    \bar{y}^2 + \bar{h}(\bar{x})\bar{y} = \bar{g}(\bar{x}).
\end{equation*} 
Here $\bar{x}$, $\bar{y}$ are the images of $x$, $y$ in the function field of $\bar{C}$, and $\bar{g}$ (resp. $\bar{h}$) are the images of $g$ (resp. $h$) in $\Ffq[\bar{x}]$.
\end{proposition}

\begin{proof}
This is \cite[Proposition 3.1]{BornerBouwWewers2017}.
\end{proof}

In the determination of the conductor of the Frey hyperelliptic curves we need be able to show when $\bar{C}$ is semistable. We do have the following \textit{double root criterion} (see for instance, \cite[Corollary 3.4, Lemma 3.7]{BornerBouwWewers2017}).

\begin{proposition}\label{prop:double_root_criterion}
The curve $\bar{C}$ is semistable if and only if
\begin{center}
    $\begin{cases}
    \bar{h}\neq 0~and~\gcd(\bar{h}, \bar{h}^\prime, \bar{g}^\prime)=1, ~when~\ch(\Ffq)=2,  \\
    \bar{f}~\text{has at most double roots}, ~when~\ch(\Ffq)\neq 2,
    \end{cases}
    $
\end{center}
where $g,h$ as in \eqref{eq:long_weierstrass_model} with $g^\prime, h^\prime$ their formal derivatives and $f$ as in \eqref{eq:short_weierstrass_model}.
\end{proposition}

We also recall the \textit{semistable reduction theorem} by Deligne and Mumford \cite{DeligneMumford69}.

\begin{theorem}[Deligne-Mumford]\label{thm:semistable_reduction}
There exists a finite extension $L$ of $\Kq$ over the curve $C$ has semistable reduction at $\fq$.
\end{theorem}

We denote by $J$ be the Jacobian of the curve $C$. 

\begin{definition}
An abelian variety over $\Kq$ has \textit{semistable reduction} if and only if the linear part of the special fiber of the connected component of its N\'{e}ron model is an algebraic torus. Furthermore, if its toric rank is positive, we say it has \textit{multiplicative reduction}, otherwise \textit{good reduction}.
\end{definition}

\begin{theorem}\label{thm:multiplicative_Jacobian}
Let $C$ be a curve over $\Kq$ and let $J$ be the Jacobian of $C$. Then $C/\Kq$ has semistable reduction if and only if $J/\Kq$ has semistable reduction. Furthermore, if $C/\Kq$ has bad semistable reduction with a model $\calC_\fq$ that has integral special fiber, then $J/\Kq$ has multiplicative reduction.
\end{theorem}

\begin{proof}
This is \cite[Theorem 3.8]{ChenKoutsianas1}. The proof is based on \cite[Theorem 2.4]{DeligneMumford69} and \cite[Lemma 3.3.5]{Romagny} (see also \cite[Section 7.5]{Liu-book}).
\end{proof}

From the functioriality of the Jacobian $J$ of $C$ we have the following known result \cite[Proposition 9.5.6]{NeronModelsBook90}.

\begin{proposition}\label{prop:good_reduction_Jacobian}
Let $C$ be a curve over $\Kq$ and let $J$ be the Jacobian of $C$. If $C$ has good reduction then $J$ has good reduction.
\end{proposition}

From Serre-Tate \cite[Theorem 1]{SerreTate} we have a Galois representation characterization when $J$ has good reduction. We denote by $J[n]$ the $n$-torsion subgroup of $J$ where $n$ is a positive integer. We also denote by $T_\ell(J)$ the $\ell$-adic Tate module of $J$ where $\ell$ is a rational prime. 

\begin{theorem}[Serre-Tate]\label{thm:Serre_Tate}
Let $C$ be a curve over $\Kq$ and let $J$ be the Jacobian of $C$. Suppose $\ell$ is a rational prime such that $\fq\nmid \ell$. Then the following are equivalent:
\begin{itemize}
    \item $J$ has good reduction.
    \item $J[m]$ is unramified for all positive integers $m$ such that $\fq\nmid m$.
    \item $T_\ell(J)$ is unramified.
\end{itemize}
\end{theorem}

Finally, we give a last definition that we need later.

\begin{definition}
Suppose $F$ is a hyperelliptic model over $\OKq$ of a hyperelliptic curve $C$ over $\Kq$ as in \eqref{eq:long_weierstrass_model}. If $h(x) = b_{g_C}x^{g_C} + \cdots + b_0$ and $g(x) = a_n x^n + \cdots + a_0$ where $n=2g_c+1$, then we define the \textit{valuation vectors over $\Kq$} of this hyperelliptic model as the pair of vectors
\begin{equation*}
    (\vfq(a_0), \cdots, \vfq(a_{n})) \qquad (\vfq(b_0), \cdots, \vfq(b_{g_C})).
\end{equation*}
\end{definition}

\subsection{Weil-Deligne representations and local inertia types}\label{sec:Weil_Deligne_representations}

In this section we recall basic definitions and material about Weil-Deligne representations and local inertia types. Our main reference is \cite{DembeleFreitasVoight22}. We also recommend the online notes by Wiese \cite{Wiese12}.

Let $\pifq$ be an uniformizer of $\fq$, $\Ffq$ the residue field of $\OKq$ with $\#\Ffq=q$, $\vfq$ the valuation of $\Kq$ with $\vfq(\pifq)=1$ and $|\cdot|_{\fq}:~\Kq^\times\rightarrow\R^\times_{>0}$ the normalized absolute value with respect to $\vfq$. We denote by $\Kqun\subset \bKq$ the maximal unramified extension of $\Kq$.

Let $\WKq$ be the Weil group of $\Kq$ that is the subgroup of $\GalKq$ which consists of the elements whose image in $\Gal(\bFfq/\Ffq)$ is a power of the Frobenius automorphism. We denote by $\IKq$ the inertia subgroup of $\WKq$ which fits on the exact sequence
\begin{equation*}
    1\rightarrow\IKq\rightarrow\WKq\rightarrow\Z\rightarrow 1.
\end{equation*}

A group homomorphism $\chi:~\WKq\rightarrow\C^\times$ with open kernel is called \textit{quasicharacter} and if $|\chi(g)|=1$ for all $g\in\WKq$ then we call $\chi$ \textit{(unitary) character}. Let $\omega:~\WKq\rightarrow\C^\times$ be the quasicharacter corresponding to the norm quasicharacter $|\cdot|_{\fq}$, so that $\omega(g)=q^{-a(g)}$ where the image of $g$ in $\Gal(\bFfq/\Ffq)$ is equal to $\Frobp^{a(g)}$ and $\Frobp$ is the geometric Frobenius element of $\WKqab$ with $\WKqab$ the maximal abelian quotient of $\WKq$. We denote by $\Art_{\Kq}:\Kq^{\times}\rightarrow \WKqab$ the Artin recipocity map from local class field theory.

\begin{definition}
A \textit{Weil-Deligne representation} is a pair $(\rho, N)$ such that
\begin{itemize}
    \item $\rho:~\WKq\rightarrow\GL_n(\C)$ is homomorphism with open kernel
    \item $N\in M_n(\C)$ is nilpotent and it holds 
    \begin{equation*}
        \rho(g)N\rho(g)^{-1}=\omega(g)N,\qquad~\text{for all }g\in\WKq. 
    \end{equation*}
\end{itemize}
Two Weil-Deligne representations $(\rho, N)$ and $(\rho^\prime, N^\prime)$ are called \textit{isomorphic} if there exists $P\in\GL_n(\C)$ such that $\rho^\prime(g)=P\rho(g)P^{-1}$ for all $g\in\WKq$ and $N^\prime=PNP^{-1}$.

Let $E_\lambda$ be a finite extension of $\Q_\ell$ where $q \not= \ell$. A continuous representation $\rho:~\WKq\rightarrow\GL_n(E_\lambda)$ is called a $\lambda$-adic representation (also called a $\ell$-adic representation as $E_\lambda \hookrightarrow \overline{\Q}_\ell$).
\end{definition}

\begin{definition}\label{def:WD_type}
An \textit{inertial Weil-Deligne type (or WD-type)} is an isomorphism class of Weil-Deligne representations restricted to inertia.
\end{definition}

The following proposition relates $\ell$-adic representations and Weil-Deligne representations \cite{Tate-Corvallis, Wiese12} (see also \cite[Section 8]{Deligne73}). To make the transition from $E_\lambda$ to $\C$ in the definition of Weil-Deligne representation, one uses the isomorphism $\overline{\Q}_\ell \cong \C$.
\begin{proposition}\label{prop:l_adic_weil_deligne} Let $E_\lambda$ be a finite extension of $\Q_\ell$. There exists a bijection between Weil-Deligne representations and $\ell$-adic representations $\WKq\rightarrow\GL_n(E_\lambda)$.

In particular, if $\rho:~\GalKq\rightarrow\GL_n(E_\lambda)$ is a $\ell$-adic representation then $\rho$ gives a unique Weil-Deligne representation which we denote by $WD(\rho)$. 
\end{proposition}

For the case $n=2$ there exists an explicit classification of Weil-Deligne representations which we explain below.

\subsubsection{Classification of Weil-Deligne representations for $n=2$}\label{sec:Weil_Deligne_rep}

Every $2$-dimensional Weil-Deligne representation is isomorphic to one of the representations below.

\vspace{0.3cm}
\textbf{Principal series:} Let $\chi_1,\chi_2:\WKq\rightarrow\C^\times$ be two quasicharacters such that $\chi_1\chi_2\neq\omega^{\pm 1}$. Let
\begin{equation*}
    \PS(\chi_1, \chi_2):=\chi_1\oplus\chi_2.
\end{equation*}
The Weil-Deligne representation $(\PS(\chi_1, \chi_2), 0)$ is called the \textit{principal series} representation associated to $\chi_1,\chi_2$. 

The conductor exponent of the principal series representation is given by
\begin{equation*}
    \condexp(\PS(\chi_1, \chi_2)) = \condexp(\chi_1) + \condexp(\chi_2).
\end{equation*}

\vspace{0.3cm}
\textbf{Special or Steinberg representations:} Let $\chi:\WKq\rightarrow\C^\times$ be a quasicharacter. Let
\begin{equation*}
    \St(\chi) := \chi\omega\oplus\chi \qquad \text{and} \qquad N =\begin{pmatrix} 0 & 1 \\ 0 & 0 \end{pmatrix}.
\end{equation*}
The Weil-Deligne representation $(\St(\chi), N)$ is called the \textit{special or Steinberg} representation associated to $\chi$.

The conductor exponent of the the special/Steinberg representation is given by
\begin{equation*}
    \condexp(\St(\chi)) = 
    \begin{cases}
        2\condexp(\chi), & \text{if }\chi\text{ is ramified},\\
        1, & \text{otherwise}.
    \end{cases}
\end{equation*}

\vspace{0.3cm}
\textbf{Supercuspidal representations:} Suppose $\rho$ is an irreducible $2$-dimensional representation of $\WKq$. Then the Weil-Deligne representation $(\rho, 0)$ is called \textit{supercuspidal}. The family of supercuspidal Weil-Deligne representations is divided in two subfamilies according to the projective image of $\rho$ in $\PGL_2(\C)$. When the projective image of $\rho$ is dihedral then $\rho$ is called \textit{nonexceptional supercuspidal representation}, otherwise $\rho$ is called \textit{exceptional supercuspidal representation} with projective image isomorphic to $A_4$ or $S_4$.

\vspace{0.2cm}
\textit{Nonexceptional supercuspidal representations:} Let $F$ be a quadratic extension of $\Kq$ and $\sigma\in\WKq$ a lift of the non-trivial element of $\Gal(F/\Kq)$. For a quasicharacter $\chi:~W_F\rightarrow\C^\times$ we define the \textit{$\sigma$-conjugate} of $\chi$
\begin{align*}
    \chi^\sigma & :W_F\rightarrow\C^\times,\\
    & \chi^\sigma(g) = \chi(\sigma^{-1}g\sigma).
\end{align*}
Because $W_F$ is a normal subgroup of $\WKq$, the $\sigma$-conjugate of $\chi$ is well-defined and independent of $\sigma$.

Suppose $\chi\neq \chi^{\sigma}$. The Weil-Deligne representation $(\Ind_{W_F}^{\WKq}\chi, 0)$, where $\Ind_{W_F}^{\WKq}\chi$ is the induced representation of $\chi$ from $W_F$ to $\WKq$, is called the \textit{nonexceptional supercuspidal representation} associate to $\chi$. The assumption that $\chi\neq \chi^{\sigma}$ ensures that $\Ind_{W_F}^{\WKq}\chi$ is irreducible. If $\chi_F$ is the quadratic character of $\WKq$ corresponding to the quadratic extension $F$, then
\begin{equation*}
    \condexp(\Ind_{W_F}^{\WKq}\chi) = 
    \begin{cases}
        2\condexp(\chi), & \text{if }F/\Kq\text{ is unramified},\\
        \condexp(\chi) + \condexp(\chi_F), & \text{otherwise}.
    \end{cases}
\end{equation*}

\vspace{0.2cm}
\textit{Exceptional supercuspidal representations:} This case exists only when the residual characteristic is $2$. Let $F/\Kq$ be a tamely ramified cubic extension and $M/F$ a ramified quadratic extension. Let $\chi$ be a quasicharacter of $W_M$ with the property that there is no character $\theta:F^\times\rightarrow\C^\times$ such that $\chi\circ\Art_M=\theta\circ\Art_F\circ\Norm_{M/F}$.

For a given triple $(\chi, M, F)$ the Weil-Deligne representation $(\rho, 0)$ such that
\begin{equation*}
    \rho\mid_{W_{F}}=\Ind_{W_M}^{W_F}\chi,
\end{equation*}
is called an \textit{exceptional supercuspidal representation}. Conversely, any exceptional supercuspidal representation is uniquely deteremined by a triple $(\chi, M, F)$ \cite[Section 12.1.3]{Carayol86}.

\vspace{0.2cm}
The reason why we are interested in inertia types is the following proposition (see for instance \cite{Ulmer}).

\begin{proposition}\label{prop:condexp_padic_weil_deligne}
Let $E_\lambda$ be a finite extension of $\Q_\ell$. Suppose $\rho_\ell:~\GalKq\rightarrow\GL_n(E_\lambda)$ is a $\ell$-adic representation and $WD(\rho) = (\rho, N)$ the corresponding Weil-Deligne representation of $\rho$. Then,
\begin{equation*}
    \condexp(\rho)=\condexp(WD(\rho)).
\end{equation*}
\end{proposition}

\subsection{The basic strategy for hyperelliptic curves in families}

Let $C/K$ be a hyperelliptic curve of genus $g\geq 2$ with Jacobian $J$. One of the most important arithmetic invariants of $C$ is its \textit{conductor} which measures the arithmetic complexity of $C$. The determination of the conductor of a hyperelliptic curve is not an easy problem. However, thanks to the work of many people there are methods of computing the conductor of a {\it fixed} hyperelliptic curve \cite{Liu94a, Liu94b, DokchitserDoris19, DDMM-local} at an odd prime or in genus $2$. Recently in \cite{maistret}, these methods have been successfully applied to compute the conductor at an odd prime of the Frey abelian varieties of signature $(p,p,r)$ and $(r,r,p)$.

In our case, we want to determine the conductor of a parametric family of hyperelliptic curves.  In the elliptic curves case Tate's algorithm is the key ingredient in the determination of the conductor of a parametric family of elliptic curves. There is no analogous version of Tate's algorithm for higher genus hyperelliptic curves currently known. Therefore, the determination of the conductor of our Frey hyperelliptic curves becomes a challenging problem, especially at the prime $2$.

We assume that $J/K$ is of $\GL_2(K)$-type over $K$. Let $\fp$ be a prime of $K$ and $\rhoJp:G_K\rightarrow\GL_{2}(\Kp)$ the Galois representation arising by the action of $G_K$ on the $\fp$-adic Tate module of $J$. Let $\brhoJp$ be the residual representation of $\rhoJp$. In the modular method we need the \textit{Artin conductor $\n(\brhoJp)$ of $\brhoJp$ away from $p$} (this is also called the \textit{Serre level of $\brhoJp$}). The Artin conductor away from $p$ will only take on a finite set of possible values and will be the level of the Hilbert newforms we need to compute.

We compute the conductor of $\brhoJp$ away from $\fp$ by computing the conductor of $\rhoJp$ and then we determine the primes over which $\rhoJp$ degenerates. In general, $\condexp_{\fq}(\brhoJp)\leq\condexp_{\fq}(\rhoJp)$ and we say $\rhoJp$ \textit{degenerates at the prime $\fq$} if $\condexp_{\fq}(\brhoJp)<\condexp_{\fq}(\rhoJp)$.

The general method of computing the Artin conductor of $\brhoJp$ in \cite{Chen-2022-xhyper, ChenKoutsianas1} has the following key steps:
\begin{itemize}
    \item For each prime $\fq$ we explicitly compute the extension $L/\Kq$ over $C$ attains semistable reduction. From Theorem \ref{thm:semistable_reduction} we know that such an extension exists and it is of finite index.
    
    \item On the assumption that we have established modularity of $J/K$, the $\rho_{J,\lambda}$ lie in a strictly compatible system of $\lambda$-adic representation, as $\lambda$ varies over the primes of $K$ not dividing the prime $q$ below $\Fq$. This means $\rho_{J,\lambda} \mid_{D_\Fq}$ have the same associated Weil-Deligne representation and hence the same conductor dependent only on $\rho_{J,\lambda}\mid_{I_\Fq}$. Recalling
    \begin{equation*}
    \rho_{J,\ell} = \oplus_{\lambda \mid \ell} \rho_{J,\lambda}, 
    \end{equation*} 
    where $\rho_{J,\ell}$ is the Galois representation obtained by the action of $G_K$ on $T_\ell(J)$. We see from Theorem~\ref{thm:Serre_Tate} that $J/K$ has good reduction at $\fq$ if and only if any $\rho_{J,\lambda} \mid_{D_\Fq}$ is unramified. If this is the case, then all the $\rho_{J,\lambda} \mid_{D_\Fq}$ are unramified.

    \item As long as $L$ is known we determine the inertia type of $\rho_{J,\lambda} \mid_{D_\fq}$, where $D_{\fq}$ is the decomposition group of $\fq$. We explicitly describe all the objects that characterize each inertia type as we have explained in section \ref{sec:Weil_Deligne_rep}.
    
    \item From the conductor exponent formulas of the inertia type in section \ref{sec:Weil_Deligne_rep} we compute the conductor exponent of $\rhoJp$ at $\fq$ using Proposition \ref{prop:condexp_padic_weil_deligne} and specializing to $\lambda = \Fp$.
    
    \item We determine the primes $\fq$ where the representation $\rhoJp$ degenerates and we compute the Artin conductor of $\brhoJp$ away from $p$.
\end{itemize}
In \cite{ChenKoutsianas1}, the strategy of the first step is used, but we complete the conductor calculation without making extensive mention of inertial types. Furthermore, we employ a different technique at the prime $\Fq_2$ for $J_r^-$: using the relation in Example~\ref{relation-1}, we ``propagate'' the computation of conductors away from $\Fq_r$ from the Legendre elliptic curve to the Frey hyperelliptic curve, up to quadratic twist. This is significantly easier to do than determining explicit semistable models. 

In the next section, we give an explicit example of conductor computation for a Frey hyperelliptic curve using the method of inertia types.

\subsection{Example}

Let $C=C_5(a,b)$ be the Frey hyperelliptic curve of the generalized Fermat equation
\begin{equation*}
    x^5 + y^5 = 3z^p,
\end{equation*}
as we have described in Section \ref{sec:r_r_p_Frey_curves} where $(a,b,c)$ is a non-trivial primitive solution of the last equation. We denote by $J$ the Jacobian of $C$. We recall that $J/K$ is of $\GL_2(K)$-type over $K$ where $K=\Q(\zeta_5)^+$. Our goal is to compute the conductor of $\rhoJp$ and $\brhoJp$ using the method we described above. For more details about the techniques we will present below see \cite{Chen-2022-xhyper}. In particular, all the results below together with their proofs are the same as in \cite[Section 9]{Chen-2022-xhyper} but explained in the case $r=5$ for concreteness.

\begin{proposition}
Let $q$ a prime such that $q\nmid 10(a^5 + b^5)$. Then, $C$ and $J$ have good reduction over $\Q_q$.
\end{proposition}

\begin{proof}
From \eqref{Kraus-discriminant} we recall that
\begin{equation*}
    \Delta(C) = 2^85^5(a^5 + b^5)^4.
\end{equation*}
Hence, by Proposition \ref{prop:good_reduction_disc} we get that $C$ has good reduction over $\Q_q$.
\end{proof}

Let $\fqf$ be the unique prime of $K$ above $5$.

\begin{proposition}\label{prop:5_5_p_field_good_reduction}
Let $L/\Kqf$ be a finite extension with ramification index $4$. If $5\nmid a + b$ then $C$ and $J$ have good reduction over $L$. In particular, there is no a totally ramified quadratic extension of $\Kqf$ over $C$ and $J$ attain good reduction.
\end{proposition}

\begin{proof}
Let $L/\Kqf$ be a totally ramified extension of degree $4$ with ring of integers $\OKqf$ and uniformizer $\pi$. We recall from Section \ref{sec:r_r_p_Frey_curves} that an affine model of $C$ is given by
\begin{equation*}
    C:~y^2= x^5 + 5abx^3 + 5a^2b^2x + b^5 - a^5.
\end{equation*}
We make the change of variables $x\mapsto\pi^2x - (b-a)$, $y\mapsto \pi^5y$ and divide by $\pi^{10}$ to get the model $C^\prime$ over $L$
\begin{align*}
    C^\prime:~y^2 & = x^5 + 5\frac{a - b}{\pi^2}x^4 + 5\frac{2a^2 - 3ab + 2b^2}{\pi^4}x^3 + 5\frac{2a^3 - 3a^2b + 3ab^2 - 2b^3}{\pi^6}x^2\\
    & + 5\frac{a^5 + b^5}{\pi^8(a+b)}x.
\end{align*}
Because $5\nmid a + b$ and $L$ is a totally ramified extension of degree $8$ over $\Q_5$ we get that the valuation vectors over $L$ of $C^\prime$ is 
\begin{equation*}
    (\infty, \geq 0, \geq 0, \geq 0, \geq 0, 0) \qquad (\infty, \infty, \infty).
\end{equation*}
Therefore, the equation $C^\prime$ is an integral model over $L$. From the formula \eqref{eq:discriminant_between_models} we get that $v_{\pi}(\Delta(C^\prime)) = 0$. Thus, by Proposition \ref{prop:good_reduction_disc} the curve $C$ attains good reduction over $L$.

Suppose $C$ attains good reduction over a totally ramified extension $F/\Kqf$ of degree $e$ and uniformizer $\pi_F$. By Proposition \ref{prop:good_reduction_disc} $C$ has an odd degree hyperelliptic model $\calC$ over $F$ such that $v_{\pi_{F}}(\Delta(\calC)) = 0$. Again, by the formula \eqref{eq:discriminant_between_models} and the fact that $v_{\pi_F}(\Delta(C)) = 10e$ we get that $4\mid e$.
\end{proof}

\begin{proposition}
Let $5\nmid a+b$. Then the inertia type of $J/\Kqf$ is a principal series. Moreover, the conductor of $\rho_{J,\lambda}$ at $\fqf$ is $\fqf^2$.
\end{proposition}

\begin{proof}
From Proposition \ref{prop:5_5_p_field_good_reduction} we know that $J$ obtains good reduction over any extension $L/\Kqf$ of ramification index $4$. Therefore, $\rho_{J,\lambda} \mid_{D_{\fqf}}$ is unramified over $L$ so the inertia type of $\rho_{J,\lambda}$ is not Steinberg at $\fqf$. Moreover, the inertia type is principal series and not supercuspidal because we can choose $L/\Kqf$ to be abelian. By class field theory such an (unique) extension exists because $4\mid \mathbb{F}_{\fqf}^\times$, where $\mathbb{F}_{\fqf}$ is the residue field associated to $\fqf$.

Let $L/\Kqf$ be the totally ramified abelian extension of degree $4$. The extension $L$ corresponds to a character $\chi$ of conductor $\fqf$. From \cite[Lemma 7.13]{Chen-2022-xhyper} we have that $\det(\rho_{J,\lambda})$ is the cyclotomic $\ell$-adic character, so $(\rhoJp\otimes\overline{\Q}_\ell)\mid_{I_{\fqf}}$ is non-trivial and isomorphic\footnote{This follows from the classification of finite subgroups of $\GL_2(\C)$ \cite{NguyenPutTop08}.} to $(\chi^k\otimes\chi^{-k})\mid_{I_{\fqf}}$ for some $k\in\Z$. From above $\chi^{k}\mid_{I_{\fqf}}$ can not be trivial, so $\chi^k$ is of conductor $\fqf$. Hence, from the formulas in Section \ref{sec:Weil_Deligne_rep} we get that the conductor of $\rho_{J,\lambda}$ at $\fqf$ is $\fqf^2$.
\end{proof}

Working as above we can compute the conductor of $\rhoJp$ at any prime $\fq$ of $K$. In particular, in \cite[Theorem 9.16]{Chen-2022-xhyper} the authors prove the following.

\begin{theorem}
Assume $a\equiv 0\pmod{2}$ and $b\equiv 1\pmod{4}$. Then the conductor $\mathcal{N}$ of $\rho_{J,\lambda}$ is given by
\begin{equation*}
    \mathcal{N} = 2^2\cdot\fqf^2\cdot\mathfrak{n},
\end{equation*}
where $\mathfrak{n}$ is the squarefree product of all prime ideal dividing $a^5 + b^5$ which are coprime to $10$.
\end{theorem}

Once we have computed the conductor of $\rhoJp$ we have to determine the Artin conductor of $\brhoJp$ away from $p$. Specializing \cite[Theorem 11.6]{Chen-2022-xhyper} to $r = 5$, we obtain the following (see also Section~\ref{finite-flat-section}).

\begin{theorem}\label{thm:finiteness_of_rhoJp}
Let $p$ be a rational prime number. Let $\fq$ be a prime in $K$ not dividing $10$ such that $v_{\fq}(a^5 + b^5)\equiv 0\pmod{p}$. We have the following conclusions:
\begin{itemize}
    \item If $\fq$ does not divide $p$, then $\brhoJp$ is unramified at $\fq$ for all $\fp\mid p$ in K;
    
    \item If $\fq$ divides $p$, then $\brhoJp$ is finite at $\fq$ for all $\fp\mid p$ in K.
\end{itemize}
\end{theorem}

An immediate consequence of Theorem \ref{thm:finiteness_of_rhoJp} is the determination of the Artin conductor away from $p$ of $\brhoJp$ \cite[Proposition 12.2]{Chen-2022-xhyper}.

\begin{proposition}
The Artin conductor away from $p$ of $\brhoJp$ divides $2^2\fqf^2\mathfrak{n}_3$ where $\mathfrak{n}_3$ is the squarefree product of the primes $\fq\mid 3$.
\end{proposition}

\begin{remark}
In the result above, the Serre level is not determined exactly. However, using the refined level lowering result of Breuil-Diamond (see \cite{Chen-2022-xhyper} for a description of how this is used), we can prove modularity occurs at level $2^2 \fqf^2 \mathfrak{n}_3$, for instance, even if it does not coincide with the actual Serre level. In fact, it is possible that the Serre level would not be the best level to perform the elimination step, due to limitations in efficiency of algorithms for computing Hilbert modular forms. For example, it is computationally beneficial to include the Steinberg prime $\mathfrak{n}_3$, even if it does not occur in the actual Serre level.
\end{remark}

\section{Irreducibility}

In this section, we describe the two main ways to prove irreducibility of residual Galois representations.

\subsection{The basic strategy using local methods}

Suppose we wish to prove that the residual representation attached to a Frey abelian variety $J$,
\begin{equation*}
    \rhobar_{J,\Fp} : G_K \rightarrow \GL_2(\F_\Fp),
\end{equation*}
is irreducible for large $\Fp \mid p$. Assume toward a contradiction that the representation $\rhobar_{J,\Fp}$ is reducible, then 
\begin{equation}
\label{reducible-relation}
    \rhobar_{J,\Fp} = \begin{pmatrix} \theta & \star \\
    0 & \theta'
    \end{pmatrix}. 
\end{equation}

In the local method, we bound the Artin conductors of the diagonal characters and show for large $p$, one of the diagonal characters is unramified at $\Fp$ and the other is ramified at $\Fp$ with known expression in terms of tame characters by Raynaud's results \cite{Raynaud}. This requires that $\rhobar_{J,\Fp}$ be finite flat at $\Fp$. Now assuming $J$ has potential good reduction at some small prime $\Fq$, this is shown to be incompatible with \eqref{reducible-relation} and all the properties the diagonal characters must possess.

A more general irreducibility result for abelian varieties of $\GL_2$-type can be found in \cite{BCDF1}. This often suffices for asymptotic bounds on the prime exponent $p$.

\subsection{The basic strategy for elliptic curves using Kamienny-Mazur}\label{Mazur-Kamienny}

When $J = E$ is an elliptic curve, we have the powerful method of Kamienny-Mazur at our disposal which does not require a known small prime of potential good reduction to work. 

The following is Mazur's original result about irreducibility \cite{Mazur-isogeny}. 

\begin{theorem}
Let $E$ be an elliptic curve over $\Q$. Then $\rhobar_{E,p}$ is irreducible for $p > 163$.
\end{theorem}

Mazur's method involves studying the modular curve $X_0(p)$ and its Jacobian $J_0(p)$. In particular, a necessary precondition to apply the method is the existence of a non-trivial quotient of $J_0(p)$ with finitely many $\Q$-rational points, which was provided by the Eisenstein quotient in Mazur's original proof. The method then shows that the elliptic curve $E$ with reducible $\rhobar_{E,p}$ must have potentially good reduction at all primes $\not= 2, p$.

To solve the uniform boundedness conjecture, this method was generalized by Kamienny-Mazur \cite{Kamienny-Mazur} (see Bourbaki exposition by Edixhoven \cite{Edixhoven-Torsion})  and used by Merel \cite{Merel} to prove the uniform boundedness conjecture for torsion of elliptic curves over number fields. Merel replaces the Eisenstein quotient by the winding quotient in order to check the conditions required by the Kamienny-Mazur method. Parent \cite{Parent} subsequently gave refined bounds and recently these have been made even more explicit for number fields of small degree \cite{Derickx, Derickx-Najman}.

In the modular method, we typically consider elliptic curves $E$ over $K = \Q(\zeta_r)^+$. To reduce to the uniform boundedness conjecture, one typically needs to show one of the diagonal characters can be taken to be the trivial character so $E$ with $\rhobar_{E,p}$ reducible can be twisted so it has a point of order $p$. Also, as the Frey elliptic curves we consider have additional small level structure, this can be used to refine the bounds on $p$ \cite[Appendix]{Chen-2022-xhyper}.

\subsection{Some examples}

Let $\fp_r$ be the unique prime of $K = \Q(\zeta_r)^+$ above the prime $r$.
\begin{theorem}
The representation $\rhobar_{J_r,\fp_r}$ is irreducible when restricted to $G_{\Q(\zeta_r)}$ if $r \not= 2,3,5,7$.
\end{theorem} 

\begin{proof}
We follow the proof in \cite[Remark 8.7]{Chen-2022-xhyper}. Since $\rhobar_{J_r,\fp}$ is equivalent to $\rhobar_{L,r}$ and $\rhobar_{L,r}(G_{\Q(\zeta_r)}) = \rhobar_{L,r}(G_{\Q})\cap\SL_2(\F_r)$, it suffices to prove $\rhobar_{L,r}$ is irreducible, where $L$ is the classical Frey elliptic curve 
\begin{equation*}
y^2 = x(x-a^r)(x+b^r).
\end{equation*}

As \(L\) has good or bad multiplicative reduction at~\(r\) (and~\(r\geq 7\)), it follows from~\cite[\S\S1.11-1.12 and Proposition~17]{Ser72} that if~\(\rhobar_{L,r}(G_\Q)\neq \GL_2(\F_r)\), then~\(\rhobar_{L,r}(G_\Q)\) is either contained in a Borel subgroup or in the normalizer of a Cartan subgroup of~\(\GL_2(\F_r)\). In the former case, \(L\) gives rise to a rational point on~\(Y_0(2r)\), contradicting results of Mazur and Kenku (see~\cite[Theorem~1]{Ken82}). In the latter case, using~\cite[Theorem~6.1]{Bilu-Parent} and~\cite[Proposition~2.1]{lemos} when the Cartan subgroup is split or non-split respectively, we get that~\(j(L) \in\Z\) and we conclude as in the rest of the proof of \cite[Proposition 8.4]{Chen-2022-xhyper}.
\end{proof}

\begin{theorem}
Assume $2 \nmid a+b$. Then $\rhobar_{J_7,\fp}$ is irreducible for all $\fp \mid p$ when $p \not= 2, 7$.
\end{theorem}

\begin{proof}
Let $J = J_7(a,b)$ where $(a,b,c)$ is a non-trivial primitive solution of 
\begin{equation*}
x^7 + y^7=z^p. 
\end{equation*}
From Theorem~\ref{det-cyclotomic}, $\det \rhobar_{J,\fp} = \chi_p$. Suppose $\rhobar_{J,\fp}$ is reducible, that is,
\begin{equation*}
\rhobar_{J,\fp} \simeq \begin{pmatrix} \theta & \star\\ 0 & \theta' \end{pmatrix} 
\quad \text{with} \quad \theta, \theta' : G_K \rightarrow \F_{\fp}^\times
\quad \text{satisfying} \quad \theta \theta' = \chi_p,
\end{equation*}
where $\F_{\fp}$ is the residual field of $K$ at $\fp$. For all primes $\fq \nmid p$, we have $\theta'|_{I_\fq} = \theta^{-1}|_{I_\fq}$, hence the conductor exponent of~$\theta$ and~$\theta'$ is the same at each such prime. 

The conductor $N(\rhobar_{J,\fp}) = \fq_2^2 \fq_7^2$ is determined in  \cite[Proposition 12.2]{Chen-2022-xhyper}. Since~$p\not=2,7$, it follows that the conductor of $\theta$ and $\theta'$ away from~$p$ divides $\fq_2 \fq_7$. 

The representation $\rhobar_{J,\Fp}$ is finite at all $\Fq \mid p$. Since $p$ is unramified 
in~$K$ it follows by \cite[Corollaire 3.4.4]{Raynaud}, for each~$\Fq \mid p$,  
the restriction to~$I_{\Fq}$ of the semisimplification $\rhobar_{J,\Fp}^{\text{ss}}$ is isomorphic to either 
$\chi_p|_{I_\Fq} \oplus 1$ or $\psi \oplus \psi^p$, where~$\psi$ is a fundamental character of level~$2$. Since $\theta$, $\theta'$ are valued in 
$\F_{\Fp}$ which is either $\F_{p}$ or $\F_{p^3}$ the case of fundamental 
characters of level 2 is excluded, because $\F_{p^2} \not\subseteq \F_{\Fp}$. We conclude that, for each $\Fq \mid p$, we have 
\begin{equation}\label{E:inertiaAction}
 \rhobar_{J,\Fp}^{\text{ss}}|_{I_\Fq} \simeq \theta|_{I_\Fq} \oplus \theta'|_{I_\Fq} \simeq \chi_p|_{I_\Fq} \oplus 1. 
\end{equation}

There are two cases:
\begin{enumerate}

    \item[(I)] One of $\theta$, $\theta'$ is unramified at every prime above~$p$. We can assume  (after relabeling if needed) that~$\theta$ is unramified at~$p$. The conductor of $\theta$ divides $\fq_2\fq_7$, so $\theta$ is a character of the Ray class group of modulus $\fq_2\fq_7\infty_1\infty_2\infty_3$, where $\infty_i$ denote the three places at infinity. This implies that $\theta$ is at most quadratic.
    However, we have seen that
    \begin{equation*}
    \rho_{J,\fp}|_{I_{\fq_2}} \otimes \Qbar_p \simeq \chi^k \oplus \chi^{-k},
    \end{equation*}
    with~$\chi$ of order 7. This shape is preserved after reduction since $p \neq 7$, hence $\theta|_{I_{\fq_2}}$ has order multiple of~7, a contradiction.

    \item[(II)] Both $\theta$, $\theta'$ ramify at some prime above $p$. Since~$K/\Q$ is Galois of degree 3 and~$p$ unramified in~$K$, we have that~$p$ is either inert or splits completely. From~\eqref{E:inertiaAction} 
    it follows that $p$ splits as $(p) = \Fp_1\Fp_2\Fp_3$ and one of $\theta$, $\theta'$ ramifies at two of these primes and the other at the remaining one. After relabeling if necessary, we can assume $\theta |_{I_{\Fp_i}} = \chi_p |_{I_{\Fp_i}}$ and $\theta$ is unramified at $\Fp_j$ for $j \neq i$, and so $\theta$ corresponds to a character of the Ray class group of modulus $\Fq_2\Fq_7\Fp_i\infty_1\infty_2\infty_3$. In particular, from the conductor of~$\theta$ we see that $\theta|_{I_{\Fq_2}}$ (after applying Artin's reciprocity from class field theory) factors via $(\calO_K/\Fq_2)^\times = \F_{2^3}^\times$ and $\theta|_{I_{\Fq_7}}$ factors via $(\calO_K/\Fq_7)^\times = \F_7^\times$. Thus $\theta^{42}$ is unramified away from~$\Fp_i$ (including infinity) and~$\theta^{42}|_{I_{\Fp_i}} = (\chi_p|_{I_{\Fp_i}})^{42}$. An application of \cite[Lemma 10.8]{Chen-2022-xhyper} with $S=\{\Fp_i \}$ gives~$u^{42} \equiv 1 \pmod{\Fp_i}$ for all units~$u \in \calO_K$ (note the norm map is the identity in this setting). In particular, $p \mid \Norm_{K/\Q}(\epsilon_1^{42}-1)$, where~$\epsilon_1 = -z$ is a fundamental unit of~$K$, therefore $p \in \{13, 29, 43, 127 \}$. A {\tt Magma} calculation shows that, for each of the four remaining~$p$ and $i=1,2,3$, the Ray class groups of modulus~$\Fq_2\Fq_7\Fp_i\infty_1\infty_2\infty_3$ are isomorphic to $\Z/2\Z \oplus \Z/2\Z$. Now the same argument with~$I_{\Fq_2}$ as in~(I) gives a contradiction.
\end{enumerate}
\end{proof}

Occasionally, one requires absolute irreducibility (i.e.\ irreducibility even after base change of the coefficient field of the representation). The following lemma is often used to achieve this.
\begin{lemma}
Suppose $\rhobar_{\fp} : G_K \rightarrow \GL_2(\F_\fp)$ is an odd Galois representation where $K$ is totally real and the characteristic of $\F_\fp$ is odd. Then $\rhobar_{\fp}$ is absolutely irreducible if and only if it is
irreducible.
\end{lemma}
\begin{proof}
Suppose $\rhobar_{\fp}$ is irreducible, but not absolutely irreducible. Then $\rhobar_\fp(G_K)$ lies in a non-split Cartan subgroup of $\GL_2(\F_\fp)$. Regarding $K \subseteq \R \subseteq \C$, we see complex conjugation on $\C$ gives an element $c$ of order $2$ in $G_K$. Since $\rhobar_\fp$ is odd, we must have that $\det \rhobar_\fp(c) = -1$, hence $\rhobar_\fp(c)$ cannot lie in a non-split Cartan subgroup of $\GL_2(\F_\fp)$.
\end{proof}

\section{Finite flatness and level lowering}

\label{finite-flat-section}

In this section, we state in convenient form the standard level lowering results for Hilbert newforms which are required for applications to the modular method. In addition, we state a criterion for the associated residual representations of a Frey abelian variety to be to be finite flat in the case that it arises from the Jacobian of a hyperelliptic curve.

Let $K$ be a totally real field. For a residual Galois representation 
\begin{equation*}
\rhobar_\Fp : G_K \rightarrow \GL_2(\bFFp).
\end{equation*}
We say that $\rhobar_\Fp$ is \textit{finite flat at $\Fp$} if $\rhobar_\Fp \mid_{G_{K_\Fp}}$ arises from a finite flat $\F$-group scheme over $\calO_\Fp$ where $\F$ is a finite extension of $\FFp$ (see \cite{Raynaud}).

For the mod $p$ representation of an elliptic curve $E$, the following proposition describes when it is unramified or finite flat.
\begin{proposition}
\label{finite-E}
Let $E$ be an elliptic curve over a number field $K$. Then
\begin{enumerate}
\item $\rhobar_{E,p}$ is unramified at $\Fq \nmid p$ if and only if $v_\Fq(\Delta(E/K_\Fq)) \equiv 0 \pmod p$,
\item $\rhobar_{E,p}$ is finite flat at $\Fq \mid p$ if and only if $v_\Fq(\Delta(E/K_\Fq)) \equiv 0 \pmod p$,
\end{enumerate}
where $\Delta(E/K_\Fq)$ is the minimal discriminant of $E/K_\Fq$.
\end{proposition}
\begin{proof}
  See \cite[Propositions 2.11 and 2.12]{DDT}.
\end{proof}

The exact analogue of the above proposition for Hilbert-Blumenthal abelian varieties is given in \cite{Ellenberg}. However, it requires the determination of a discriminantal set in order to give an equivalent condition for being unramified or finite flat. The following gives only a sufficient condition and is based on a method of Darmon given in \cite{DarmonDuke}. For Diophantine applications, a sufficient condition is usually what is used in the modular method.
\begin{theorem}\label{T:finite}
Let~\(p\) be a rational prime number.
Let~$\Fq$ be a prime in~$K$ not dividing~$2r$ such that~$v_\Fq(a^r + b^r)\equiv 0\pmod{p}$. We have the following conclusions~:
\begin{itemize}
    \item If $\Fq$ does not divide $p$, then~$\rhobar_{J_r,\Fp}$ is unramified at~$\Fq$ for all~\(\Fp\mid p\) in~\(K\);
    
    \item If $\Fq$ divides $p$, then~$\rhobar_{J_r,\Fp}$ is finite flat at~$\Fq$ for all~\(\Fp\mid p\) in~\(K\).
\end{itemize}
\end{theorem}
\begin{proof}
  See \cite{Chen-2022-xhyper}.
\end{proof}

\begin{theorem}
Suppose $\rhobar_\Fp \simeq \rhobar_{f',\Fp}$ is absolutely irreducible, finite flat at $\Fp$, and arises from a Hilbert newform $f'$ over $K$ of level $\n'$, parallel weight $2$, and trivial character. Then $\rhobar_\Fp \simeq \rhobar_{f,\Fp}$ for a Hilbert newform $f$ over $K$ of level $\n(\rhobar_\Fp)$, parallel weight $2$, and trivial character.
\end{theorem}

\begin{proof}
This follows from \cite{Fuj, Jarv, Jarv2, Raj}.
\end{proof}

\section{The elimination step}

In this section, we describe the known methods to distinguish two residual Galois representations 
\begin{align*}
    \rhobar_{1,\Fp} & : G_K \rightarrow \GL_2(\F_{\Fp_1}), \\
    \rhobar_{2,\Fp} & : G_K \rightarrow \GL_2(\F_{\Fp_2}),
\end{align*}
in order to carry out the elimination step.

\subsection{Techniques for elimination}

The first method is to compare local traces of Frobenius at an unramified prime $\Fq$ and it has its origin back to Serre \cite[p. 203]{Serre87}. Namely, if 
\begin{equation}
\label{rhobar-isom}
  \rhobar_{1,{\Fp_1}} \simeq \rhobar_{2,{\Fp_1}},
\end{equation}
are unramified at $\Fq$, then we should have
\begin{equation}
\label{trace-compare-naive}
    \Tr \rhobar_{1,{\Fp_1}}(\Frob{\Fq}) = \Tr \rhobar_{2,{\Fp_2}}(\Frob{\Fq}).
\end{equation}
However, a subtlety occurs because in the definition of the isomorphism \eqref{rhobar-isom} we mean
\begin{equation}
    \rhobar_{1,\Fp_1} \otimes \overline{\F}_p \simeq \rhobar_{2,\Fp_2} \otimes \overline{\F}_p.
\end{equation}
Hence, the comparison \eqref{trace-compare-naive} cannot be done until we have chosen an embedding of the residue fields of $\Fp_1$ and $\Fp_2$ into $\overline{\F}_p$. The correct condition to rule out an isomorphism as in \eqref{rhobar-isom} by a local comparison of traces is
\begin{equation}
\label{invariant-norm}
    p \nmid \Norm_{L/\Q}\left( \prod_{\tau \in \Gal(K_1/\Q)} (a_\Fq(\rho_{2,\Fp_2}) - a_\Fq(\rho_{1,\Fp_1})^\tau) \right),
\end{equation}
where $L$ is the compositum of $K_1$ and $K_2$ inside $\overline{\Q}$ and $a_\Fq(\rho_{i,\Fp_i}) = \Tr\rho_{i,\Fp_i}(\Frob{\Fq})$ for $i=1,2$ (see \cite{Chen-2022-xhyper} for more details). Here, the field $K_i$ denotes a field contained the field generated by the traces of Frobenius of $p$-adic representations $\rho_{i,\Fp_i}$ at unramified primes and $\rhobar_{i,\Fp_i}$ is a reduction of $\rho_{i,\Fp_i}$, where $i = 1, 2$.

\begin{lemma}
\label{resultant-eliminate}
Suppose $K_2/K_1$ is an extension of number fields, $K_2/\Q$ and $K_1/\Q$ are Galois, and we have arranged that the prime $\Fp_2$ of $K_2$ lies above the prime $\Fp_1$ of $K_1$. Let $g_i(X)$ be the characteristic polynomial of $a_\Fq(\rho_{i,\Fp_i})$ from $K_i/\Q$. Then
\begin{equation}
\Norm_{K_2/\Q}\left( \prod_{\tau \in \Gal(K_1/\Q)} (a_\Fq(\rho_{2,\Fp_2}) - a_\Fq(\rho_{1,\Fp_1})^\tau) \right) = \pm \Res \left(g_2(X),g_1(X) \right).
\end{equation}
\end{lemma}
\begin{proof}
For $\tau \in \Gal(K_1/\Q)$, we let $\tau_0 \in \Gal(K_2/\Q)$ denote a choice of lift such that $\tau_0 \mid_{K_1} = \tau$.

\begingroup
\allowdisplaybreaks
\begin{align*}
& \Norm_{K_2/\Q}\left( \prod_{\tau \in \Gal(K_1/\Q)} (a_\Fq(\rho_{2,\Fp_2}) - a_\Fq(\rho_{1,\Fp_1})^\tau) \right) \\
& = \prod_{\sigma \in \Gal(K_2,\Q)} \prod_{\tau \in \Gal(K_1/\Q)} (a_\Fq(\rho_{2,\Fp_2})^\sigma - a_\Fq(\rho_{1,\Fp_1})^{\tau \sigma}) \\
& = \prod_{\tau \in \Gal(K_1/\Q)} \prod_{\sigma \in \Gal(K_2/\Q)} (a_\Fq(\rho_{2,\Fp_2})^\sigma - a_\Fq(\rho_{1,\Fp_1})^{\tau \sigma}) \\
& = \prod_{\tau \in \Gal(K_1/\Q)} \prod_{\sigma \in \Gal(K_2/\Q)} (a_\Fq(\rho_{2,\Fp_2})^{\tau_0^{-1} \sigma} - a_\Fq(\rho_{1,\Fp_1})^{\tau}) \\
& = \prod_{\tau \in \Gal(K_1/\Q)} \prod_{\sigma \in \Gal(K_2/\Q)} (a_\Fq(\rho_{2,\Fp_2})^{\sigma} - a_\Fq(\rho_{1,\Fp_1})^{\tau}) \\
& = \prod_{\sigma \in \Gal(K_2/\Q)} \prod_{\tau \in \Gal(K_1/\Q)} (a_\Fq(\rho_{2,\Fp_2})^\sigma - a_\Fq(\rho_{1,\Fp_1})^{\tau}) \\
& = \pm \Res \left(g_2(X),g_1(X) \right),
\end{align*}
\endgroup
where the last equality follows from the characteristic polynomials $g_i(X)$ being monic.
\end{proof}

For a fixed $L$ and primes $\Fp_1$ and $\Fp_2$, we can test \eqref{trace-compare-naive} directly, which is called the refined elimination method. However, it requires computing in the field $L$ which can have large degree.

If the isomorphism \eqref{rhobar-isom} arises from level lowering $\rhobar_{1,\Fp_1}$ to $\rhobar_{2,\Fp_2}$, another powerful method is to use refined level lowering result of Breuil-Diamond \cite{BreuilDiamond} which can be used to show that $K_2 \supseteq K_1$ for some choice of $\rhobar_{2,\Fp_2}$ \cite{Chen-2022-xhyper}.

In the case that the representations are ramified at $\Fq$, we can try to show
\begin{align}
   \rhobar_{1,\Fp_1} \mid_{I_\Fq} \not\simeq \rhobar_{2,\Fp_2} \mid_{I_\Fq}.
\end{align}
This is more subtle to carry out and can involve a comparison of conductors, the sizes of the images of inertia (when finite), or inertial fields, see for instance \cite{BCDF1, BCDDF}.

Finally, an idea which was proposed in Darmon's original program occurs naturally when $\rhobar_{1,\Fp_1}$ arises from one of his (even) Frey varieties: they are typically $\Q$-forms and have reducible mod $\Fq_r$-representation. 

Recall a Hilbert newform $f$ over $K$ is a \textit{$\Q$-form} if its coefficients satisfy
\begin{equation}
    a_{\Fq}(f)^\sigma = a_{\Fq^\sigma}(f),
\end{equation}
for all primes $\Fq$ of $K$ not dividing the level of $f$ and all $\sigma \in G_K$.

By a Chebotarev argument, we can rule \eqref{rhobar-isom} for large $p$ unless $\rhobar_{2,\Fp_2}$ arises from a Hilbert newform which is a $\Q$-form and has reducible mod $\Fq_r$ representation. Darmon shows that such forms often do not exist depending the class group of $\Q(\zeta_r)$ \cite{DarmonDuke}.

\subsection{Some examples}

\begin{example}
Suppose $(a,b,c) \in \Z^3$ is a non-trivial primitive solution to 
\begin{equation}
   x^ 7 + y^7 = z^p
\end{equation}
such that $2 \nmid a+b$ and $7 \mid a + b$. Then 
\begin{equation}
\label{rhobar-lower}
  \rhobar_{J_7,\Fp} \simeq \rhobar_{f,\mathfrak{P}}
\end{equation}
for a Hilbert newform of level $\Fq_2^2  \Fq_7$, trivial character, and parallel weight $2$ over $K = \Q(\zeta_7)^+$. The space of such forms is one dimensional, so there is a unique such form $f$ and its field of coefficients is $K$ so $L = K$.

For $\Fq \nmid 2 \cdot 7$, we have that \eqref{invariant-norm} holds. Let $\Fq$ be a prime of $K$ lying above the prime $q$. Suppose $f_\Fq(J_7)(X)$ is the characteristic polynomial of Frobenius at the prime $\Fq$ acting on the full Tate module $T_p(J_7)$ for $p \neq q$. This has degree $6$ and factors over $K$ as
\begin{align}
  & g_\Fq(J_7)(X) = \sum_{i=0}^6 a_i X^i \\
  & = \prod_{\sigma \in \Gal(K/\Q)} (X^2 - a_\Fq^\sigma X + \Norm(\Fq))
\end{align}
where the $a_\Fq(J_7)$ is the trace of Frobenius at $\Fq$ acting on $T_\Fp(J_7)$ for some prime $\Fp$ above the prime $p$ (see \cite{Chen-2022-xhyper} for more details). The characteristic polynomial of $a_\Fq(J_7)$ from $K/\Q$ can thus be recovered from $g_\Fq(J_7)(X)$ as
\begin{equation}
  g(a_\Fq(J_7))(X) = X^3 + a_5 X^2 + (a_4 - 3 \Norm(\Fq)) X - (a_3 - 2 \Norm(\Fq) a_5).
\end{equation}
For $\Fq = 3 \calO_K$ above $q = 3$ and $3 \nmid a + b$, we have that
\begin{align}
  g_{\Fq}(J_7)(X) & = X^6 + 81 X^4 + 2187 X^2 + 19683 = (X^2 + 27)^3 \\
  g(a_\Fq(J_7))(X) & = X^3 \\
  g(a_\Fq(f))(X) & = (X+4)^3.
\end{align}
The resultant of the last two polynomials is
\begin{equation}
  \Res(g(a_\Fq(J_7))(X), g(a_\Fq(f))(X)) = - 2^{18}.
\end{equation}
Thus, by \eqref{invariant-norm} and Lemma~\ref{resultant-eliminate}, the isomorphism \eqref{rhobar-lower} does not hold if $p \not= 2, 3$.
\end{example}

\begin{example}
Suppose $(a,b,c) \in \Z^3$ is a non-trivial primitive solution to 
\begin{equation}
   x^ 7 + y^7 = 3 z^p
\end{equation}
such that $2 \nmid a+b$ and $7 \mid a + b$. Then 
\begin{equation}
  \rhobar_{J_7,\Fp} \simeq \rhobar_{f,\mathfrak{P}}
\end{equation}
for a Hilbert newform of level $\Fq_2^2 \Fq_3 \Fq_7$, trivial character, and parallel weight $2$ over $K = \Q(\zeta_7)^+$. The space of such forms has dimension $104$ and there are $19$ Galois conjugacy classes of forms.

Only the forms numbered $10$, $11$, $12$, $16$, $17$, and $19$ in {\tt Magma} have field of coefficients containing $K$, hence the other forms are automatically eliminated by the results in \cite{Chen-2022-xhyper}.
\end{example}

\section{The multi-Frey method}

In this section, we explain the \textit{multi-Frey method} and how it can be used to give a complete resolution of a generalized Fermat equation by patching together information from different Frey abelian varieties. This is particularly useful when one Frey abelian variety fails to give a complete resolution in all congruence classes of solutions, but moving to a different Frey abelian variety achieves a resolution in that congruence class.

We also give a rough classification of the types of solutions to help organize the patching arguments.

\subsection{First and second case solutions}

Let $(a,b,c) \in \Z^3$ be a solution to a generalized Fermat equation \eqref{general-equ} and $J = J_{a,b,c}/K$ an associated Frey abelian variety defined over $K$. Given a prime $\Fq$ of $K$, let $f_{\Fq}(\rho_{J,p})$ be the conductor exponent of $\rho_{J,p}$ at $\Fq$. 

We order the solutions by increasing conductor exponent $f_\Fq(\rho_{J,p})$. This notion is relative to the Frey abelian variety $J$ and the prime $\Fq$. A solution $(a,b,c)$ which gives the smallest conductor exponent $f_\Fq(\rho_{J,p})$ possible is called a \textit{first case solution}. A solution $(a,b,c)$ which gives a conductor exponent $f_\Fq(\rho_{J,p})$ strictly bigger than the smallest value possible is called a \textit{second case solution}.

\begin{example}
Let $(a,b,c)$ be a solution to
\begin{equation*}
    x^p + y^p = z^p.
\end{equation*}
Relative to the Frey elliptic curve 
\begin{equation*}
    E_{a,b,c} : y^2 = x(x-a^p)(x+b^p)
\end{equation*}
and a prime $q$, the solution $(a,b,c)$ is a first case solution if $q \nmid abc$ and a second case solution if $q \mid abc$.
\end{example}

\begin{example}
Let $(a,b,c)$ be a solution to
\begin{equation*}
    x^p + y^p = z^r
\end{equation*}
where $r \ge 3$. Relative to the Frey abelian variety $J_r^\pm(a,b,c)$ and the prime $\Fq_r$ of $K = \Q(\zeta_r)^+$ above the prime $r$, the solution $(a,b,c)$ is a first case
solution if $r \mid a + b$ and a second case solution if $r \nmid a + b$.
\end{example}

\subsection{Obstructive solutions}

A solution $(a_0,b_0,c_0)$ is \textit{obstructive} relative to a Frey abelian variety $J_0 = J_{a_0,b_0,c_0}$ being used in the modular method if $N(\rho_{J_0,p}) = N(\rhobar_{J,p})$, that is, its conductor as an abelian variety of $\GL_2$-type occurs at the Serre level of the residual representation of the Frey abelian variety $J$.

In such a case, $J_0$ is modular and gives rise to a Hilbert newform $f_0 = f_{a_0,b_0,c_0}$ at level $N(\rhobar_{J,p})$. A local comparison of traces or residual representations will fail to eliminate the Hilbert newform $f_0$. 

Currently, there are two methods to eliminate such forms $f_0$:
\begin{enumerate}
    \item Distinguish $\rhobar_{J,p}$ and $\rhobar_{J_0,p} \simeq \rhobar_{f_0,p}$ by restriction to inertia subgroup at some prime $\Fq$ of $K$.
    \item On the assumption $J_0$ is an elliptic curve and has complex multiplication, apply Mazur's method in the form of Darmon-Merel \cite{DarmonMerel} or Ellenberg \cite{Ellenberg-Q-curve}.
\end{enumerate}

\subsection{The basic strategy of patching}

The multi-Frey method was introduced in \cite{Siksek-multi-Frey}. It works by considering multiple Frey abelian varieties $J^1, J^2, \ldots$ attached to a solution $(a,b,c)$. 

We start by applying the modular method using $J^1$. If this succeeds, we are done. If it fails, using $J^1$ may still succeed in certain congruence classes of $(a,b,c)$. For the congruence classes which cannot be dealt with using $J^1$, we apply the modular method using $J^2$. We continue in this way until all possible congruence classes of $(a,b,c)$ are covered.

At each step, the failure of applying the modular method using $J^i$ may be due to different reasons. Typically, this will be due to the inability to prove irreducibility, compute Hilbert newforms, eliminate a Hilbert newform, or occasionally absence of the required modularity statement.

Originally, the use of the multi-Frey method was to make the elimination step more efficient computationally, that is, the use of a single Frey abelian variety would in principle still work but be computationally slower. In later works, the use of multiple Frey abelian varieties has been shown to be more essential to overcome multiple reasons for the failure of applying the modular method. Hence, we view the multi-Frey method as a kind of ``patching argument'' to give a complete resolution of a family of generalized Fermat equations.

\subsection{Some examples}
\begin{example}
In the proof of Theorem~\ref{T:main5}, the patching argument can be summarized as:
\begin{enumerate}
    \item The case $2 \nmid ab, 5 \mid ab$ uses the Frey hyperelliptic curve $C_5^+(a,b,c)$. When $5 \nmid ab$, there is no known modularity lifting result which can be applied to prove $J_5^+(a,b,c)$ is modular.
    \item The case $2 \mid ab, 5 \nmid ab$ uses the Frey hyperelliptic curve $C_5^-(a,b,c)$.
\end{enumerate}

\end{example}

\begin{example}
In the proof for $r = 11$ of Theorem~\ref{T:main11}, the patching argument can be summarized as:
\begin{enumerate}
    \item The case $2 \mid a + b$ uses the Frey elliptic curve $F = F_{a,b}$ defined over $K = \Q(\zeta_{11})^+$. The case $2 \nmid a + b$ cannot be completed using $F$ due to the inability to compute Hilbert newforms at the required levels.
    \item The case $2 \nmid a + b$ uses the Frey hyperelliptic curve $C_{11}(a,b,c)$ defined over $K = \Q(\zeta_{11})^+$.
\end{enumerate}

\end{example}

\section{Survey of current results}

In this section, we survey the results from recent progress on Darmon's program. There are other surveys \cite{beal-survey, BCDY-survey, Kraus-survey, Ratcliffe} which cover cases approachable by Frey elliptic curves, so we will mainly focus on the new cases which use Frey elliptic curves or Frey abelian varieties of higher dimension over totally real fields.

\subsection{Signature \texorpdfstring{$(p,p,r)$}{(p,p,r)}}

An earlier result of Billerey-Chen-Dieulefait-Freitas \cite{BCDF1} proves the following.
\begin{theorem}
Let $r$ be  regular prime. Then for prime exponent $p$ sufficiently large compared to $r$, then
the equation
\begin{equation}\label{equ:ppr}
  x^p + y^p = z^r,
\end{equation}
has no non-trivial primitive solutions $(a,b,c) \in \Z^3$ such that $r \mid ab$ and $2 \nmid ab$.
\end{theorem}
The condition $2 \nmid ab$ is to ensure a prime a potential good reduction in order to establish irreducibility of residual representations attached to the Frey abelian variety $J_r^+$ which is used in the proof.

The main result of Chen-Koutsianas \cite{ChenKoutsianas1} is the following theorem.
\begin{theorem}\label{T:main5}
For $n \ge 3$, there are no non-trivial primitive solutions $(a,b,c) \in \Z^3$ to the equation 
\begin{equation} \label{main-equ-ppr}
  x^n + y^n = z^5,
\end{equation}
in the cases
\begin{enumerate}
\item[I.] $2 \nmid ab$ and $5 \mid ab$, or
\item[II.] $2 \mid ab$ and $5 \nmid ab$.
\end{enumerate} 
\end{theorem}
The above result can be interpreted as solving \eqref{main-equ-ppr} in the first case with respect to the primes $2$ and $5$, except that the case $2 \mid ab, 5 \mid ab$ is missing due to the inability to prove irreducibility of residual Galois representations attached to $J_5^\pm(a,b,c)$ in these congruence classes. In the next example, this issue is resolved because there is a Frey elliptic curve one can use, and such irreducibility results can be proven for elliptic curves by Kamienny-Mazur (see Section \ref{Mazur-Kamienny}).

When trying to solve~\eqref{equ:ppr} for $r = 5$, the trivial primitive solutions with $c=0$ pose an obstruction. Also, irreducibility cannot be proven in the case $10 \mid ab$. In \cite{ChenKoutsianas1}, this obstruction for the Frey variety $J_r^-(a,b,c)$ is distilled into the following statement for $r = 5$.

\begin{theorem}\label{thm:eliminate-to-CM}
For prime $p$ sufficiently large, any non-trivial primitive solution $(a,b,c) \in \Z^3$ to the equation~\eqref{main-equ-ppr} gives rise to a residual Frey representation $\rhobar_{J_5^-(a,b,c),\Fp}$ such that $\rhobar_{J_5^-(a,b,c),\Fp}$ is reducible or is a quadratic twist of $\rhobar_{J_5(\alpha, -\alpha, 0),\Fp}$ for some $\alpha \in \Z \backslash \left\{ 0 \right\}$, where $\Fp$ is any choice of prime of $K = \Q(\zeta_5)^+$ above $p$.
\end{theorem}

Since any quadratic twist of $J_5(\alpha, -\alpha, 0)$ has CM by $\Q(\zeta_5)$, the residual representation has small image, so Darmon's conjectural ideas can now kick in.

\begin{corollary}\label{assume-conj}
Assuming Conjecture 4.1 in \cite{DarmonDuke}, there are no non-trivial primitive solutions to
\begin{equation*}
    x^p + y^p = z^5,
\end{equation*}
for prime $p$ sufficiently large.
\end{corollary}

We also remark that we only need Conjecture 4.1 in the Borel case when $10 \mid ab$ as irreducibility in the other cases were established in the course of proving Theorem~\ref{T:main5}.

\subsection{Signature \texorpdfstring{$(r,r,p)$}{(r,r,p)}}

The known results are richer and more varied due to the existence of both Frey elliptic curves and Frey abelian varieties in this signature.

Billerey \cite{Billerey07} studies the case $r=5$ and $d=2^\alpha3^\beta5^\gamma$ with $0\leq\alpha,\beta,\gamma\leq 4$. In Dieulefait--Freitas~\cite{DF2,DF1} and Freitas~\cite{Freitas15} introduced several Frey elliptic curves attached to Fermat equations of signature $(r,r,p)$ and used to study the equation
\begin{equation}
  x^r + y^r = dz^p, \qquad xyz \ne 0, \qquad \gcd(x,y,z) = 1,
  \label{E:rrp}
\end{equation}
with $r$ a fixed prime, $d$ a fixed positive integer and $p$ allowed to vary. Using a refined multi-Frey elliptic curve approach, \eqref{E:rrp} was solved for $r=5,13$ and $d=3$ for all exponents $p \geq 2$ (see~\cite[Theorems~1 and~2]{BCDF2} and the joint work with Demb\'el\'e~\cite{BCDDF}). Freitas~\cite{Freitas15} solved it  asymptotically for $r=7$ and $d=3$, i.e., for all primes $p$ sufficiently large.

One of the main consequences of Billerey-Chen-Dieulefait-Freitas \cite{Chen-2022-xhyper, BCDF2} is the following theorem.

\begin{theorem}
\label{T:main11}
For $r = 5, 7, 11, 13$ and all integers $n \geq 2$, there are no non-trivial primitive solutions $(a,b,c) \in \Z^3$ to the equation
\begin{equation}\label{main-equ}
    x^{r} + y^{r} = z^n,
\end{equation} 
such that $2 \mid a + b$ or $r \mid a + b$.
\end{theorem}

Theorem~\ref{T:main11} can be interpreted as solving \eqref{main-equ} in the first case with respect to the primes $2$ and $r$. The $2$- and $r$-adic conditions are imposed because of the obstructive solutions $\pm (0, 1, 1)$ and $ \pm (1, 0, 1)$.

We note the cases $r = 5, 13$ were proven in \cite{BCDF2} using Frey elliptic curves, and the case $r = 7$ can be achieved using Frey elliptic curves, though the use of the multi-Frey method with Kraus' Frey hyperelliptic curve makes the computation faster. The case $r = 11$ genuinely uses Frey abelian varieties of higher dimension, but we note the use of Freitas' Frey elliptic curves is still essential to bootstrap the multi-Frey method and ensure a prime of potential good reduction for the Frey abelian variety.

The following result is proven by Freitas-Najman in \cite{Freitas-Najman} using only Frey elliptic curves and a clever application of the symplectic method. It applies when $d \not= 1$ is odd and divisible only by primes $q \not\equiv 1 \pmod r$.
\begin{theorem}
For a set of prime exponents $p$ of positive density, the equation
\begin{equation} 
  x^r + y^r = d z^p
\end{equation}
has no non-trivial primitive solutions $(a,b,c) \in \Z^3$ such that $2 \mid a+b$ or $r \mid a + b$.
\end{theorem}
The trivial solutions $\pm (1,0,1), \pm (0,1,1)$ pose an essential obstruction to the method used for the above theorem which is why there is a condition on $d$. This result illustrates again the theme that each Frey abelian variety may have its advantages in certain congruence classes, modifications of the equation, or additional conditions on the exponents $p$.

When trying to solve~\eqref{E:rrp} for $d = 1$, the trivial primitive solutions with $ab=0$ pose an obstruction (for any Frey variety). For $r=3$, Chen-Siksek and Freitas solved \eqref{E:rrp} for positive density exponents $p$ \cite{ChenSiksek09, Freitas16}. This obstruction for the Frey variety $J_r(a,b)$ is distilled into the following statement for $r = 5, 7$ (the idea is similar to the proof of \cite[Theorem 14.6]{Chen-2022-xhyper}; for $r = 5$ we use \cite[Theorem 4]{BCDF2}).

\begin{theorem}\label{eliminate-to-CM}
Let $r \in \left\{ 5, 7 \right\}, d = 1$ and $p \notin \left\{ 2, 3, r \right\}$ be a prime. Then for $p$ sufficiently large, any non-trivial primitive solution $(a,b,c)$ to the equation~\eqref{E:rrp} gives rise to a residual Frey representation $\rhobar_{J_r(a,b),\Fp}$ such that $\rhobar_{J_r(a,b),\Fp}$ is a quadratic twist of $\rhobar_{J_r(0,1),\Fp}$ for any choice of prime $\Fp$ of $K = \Q(\zeta_r)^+$ above $p$.
\end{theorem}

The above theorem addresses the difficulty of eliminating all Hilbert newforms without complex multiplication, thereby allowing Darmon's conjectural ideas to kick in.

\begin{corollary}
Assuming the Cartan case of Conjecture 4.1 in \cite{DarmonDuke}, then there are no non-trivial primitive solutions to
\begin{equation*}
    x^r + y^r = z^p,
\end{equation*}
for $r = 5, 7$ and when $p$ is sufficiently large.
\end{corollary}

The above corollary is an immediate consequence of Theorem \ref{eliminate-to-CM} and the analogous result to \cite[Proposition 3.7]{DarmonDuke} for the signature $(r,r,p)$. We do not need the Borel case of the conjecture as irreducibility was ``propagated'' from Freitas' Frey elliptic curves.

For $r=3,5,7$ the Frey-Mazur conjecture over $\Q$ \cite[Conjecture 3]{Kraus-survey} implies that \eqref{E:rrp} and $d=1$ has no non-trivial primitive solutions for $p$ sufficiently large. For, $r\geq 11$ we need Frey-Mazur conjecture over $\Q(\zeta_r)^+$ and the elliptic Frey curve constructed by Freitas \cite{Freitas15}.

Finally, it is worth mentioning the method of Kraus \cite{Kraus98}. Kraus studies the case $r=3$ and $d=1$ of \eqref{E:rrp} using Frey elliptic curves and gives a practical criterion in its resolution for fix $p$. It is not difficult to generalize the method in Frey abelian varieties case for $r\geq 5$ and arbitrary $d$.

\subsection{Other directions and viewpoints}

The focus of this survey has been on generalized Fermat equations signature $(r,r,p)$ and $(p,p,r)$ with integer solutions. However, the approaches discussed also generalize to case of solutions in the ring of integers of the totally real field $K$. We mention \cite{FKS-PNAS, FreitasKrausSiksek20, FreitasSiksek15, FreitasSiksek15b, OzmanSiksek22} for an overview and \cite{Jarvis-FLT-Q2} as one of the first examples of such results.

Many of the known results for solutions over number fields are asymptotic (i.e.\ true for the prime exponent $p$ sufficiently large) \cite{KaraOzman20, Kraus19, Deconinck16, IsikKaraOzman20, Mocanu22, Mocanu22b} and rely on eliminating Hilbert newforms with coefficient field not equal to the endomorphism algebra of the Frey abelian variety. This is in turn related to the non-existence of solutions to the $S$-unit equation over $K$. When there are many solutions to the $S$-unit equation, partial results can still be obtained  with conditions on the exponent $p$ using reciprocity constraints \cite{CES} or the symplectic method \cite{freitas-siksek-1}.

Darmon's classification of Frey representations holds for prime exponents. When the exponents are not prime, often descent can be used in conjunction with the modular method to prove surprising results. We mention two results  \cite{Anni-Siksek, bennett-2p-2p-5} that would normally lie in the realm of Frey abelian varieties of higher dimension but can be studied using only Frey elliptic curves due to composite exponents.

One can consider ``general coefficients'' variants of \eqref{equ:ppr}. This has been studied in the signatures $(p,p,p)$, $(p,p,2)$, and$(p,p,3)$ where Frey elliptic curves exist \cite{Bennett-Skinner, Bennett-Vatsal-Yazdani, Kraus-Halberstadt, Kraus-Ivorra}.

If $H(x,y) \in \Z[x,y]$ is a homogeneous polynomial of degree $r$ which is $\GL_2(\Qbar)$-equivalent to the homogeneous polynomial of degrees $3,4,6,12$ whose linear factors correspond to the vertices of an equilateral triangle, tetrahedron, octahedron, and icosahedron, aka ``Klein form'', then \cite{Bennett-Dahmen, Billerey-cubic} study the equation
\begin{equation}
\label{homogeneous-equ}
  H(x,y) = C z^p,
\end{equation}
using Frey elliptic curves with constant mod $2,3,4,5$ representations, respectively.

Finally, the case $(p,q,r)$ has been studied less than the cases $(p,p,r)$ and $(r,r,p)$. Recently, Freitas-Naskr\k{e}cki-Stoll have studied the signature $(2,3,p)$ \cite{FreitasNaskreckiStoll20}.

\section{Conclusion}

The study of Diophantine equations has a long history of which the generalized Fermat equation is a prime example which is elementary to describe but hard to solve. While many results may look the same or have a similar flavour, the methods and techniques needed to solve them can vary greatly in difficulty and approaches.

Darmon's program provides a fruitful path and framework for the study of generalized Fermat equations in one varying exponent, especially in the most difficult cases, and for organizing the emerging tapestry of advances on the problem of the generalized Fermat equation. 

The construction and use of multiple Frey abelian varieties beyond those predicted by Darmon is increasingly becoming crucial for resolving new signatures, and additionally refined modularity and compatibility statements seem to be useful.

\bibliography{mybib}{}
\bibliographystyle{plain}

\end{document}